\documentclass[11pt]{article}\UseRawInputEncoding
\usepackage{amssymb,amsfonts,amsmath,latexsym,epsf,tikz,url}
\usepackage[usenames,dvipsnames]{pstricks}
\usepackage{pstricks-add}
\usepackage{epsfig}
\usepackage{pst-grad} 
\usepackage{pst-plot} 
\usepackage[space]{grffile} 
\usepackage{etoolbox} 
\makeatletter 
\patchcmd\Gread@eps{\@inputcheck#1 }{\@inputcheck"#1"\relax}{}{}
\makeatother

\newtheorem{theorem}{Theorem}[section]
\newtheorem{proposition}[theorem]{Proposition}

\newtheorem{corollary}[theorem]{Corollary}

\newtheorem{remark}[theorem]{Remark}

\newcommand{\proof}{\noindent{\bf Proof.\ }}
\newcommand{\qed}{\hfill $\square$\medskip}
\newcommand{\gst}{\gamma_{\rm st}}
\newcommand{\gw}{\gamma_{\rm w}}
\newcommand{\dst}{d_{\rm st}}

\begin{document}

\title{Strong domatic number of a graph}

\author{ Nima Ghanbari$^{1}$ \and Saeid Alikhani$^{2,}$\footnote{Corresponding author}
	}

\date{\today}

\maketitle

\begin{center}

$^1$Department of Informatics, University of Bergen, P.O. Box 7803, 5020 Bergen, Norway\\	
$^2$Department of Mathematical Sciences, Yazd University, 89195-741, Yazd, Iran\\

\medskip
\medskip
{\tt  ~~
$^{1}$Nima.Ghanbari@uib.no  ~~~  $^{2}$alikhani@yazd.ac.ir 
 }

\end{center}

\begin{abstract}
  A set $D$ of vertices of a simple graph $G=(V,E)$ is a {strong dominating set}, if for every vertex $x\in  \overline{D}=V\setminus D$ there is a vertex $y\in D$ with $xy\in E(G)$ and $\deg(x)\leq \deg(y)$. The {strong domination number} $\gst(G)$ is defined as the minimum cardinality of a strong dominating set. The strong domatic number of  $G$
 is the maximum number of strong  dominating sets into which the vertex set of $G$ can be partitioned. We initiate the study of the strong domatic number, and we present different sharp bounds on $d_{st}(G)$. In addition, we determine this parameter for some classes of graphs, such as cubic graphs of order at most $10$.
 
\end{abstract}


\medskip
\noindent{\bf AMS Subj.\ Class.}:  05C69.

\section{Introduction}

The various different domination concepts are
well-studied now, however new concepts are introduced frequently and the interest is growing
rapidly. We recommend three fundamental books \cite{domination,2} and some surveys \cite{6,5} about
domination in general. A set $D\subseteq V$ is a \emph{strong dominating set} of a simple graph  $G=(V,E)$, if for every vertex $x\in  \overline{D}=V\setminus D$ there is a vertex $y\in D$ with $xy\in E(G)$ and $\deg(x)\leq \deg(y)$. The \emph{strong domination number} $\gst(G)$ is defined as the minimum cardinality of a strong dominating set. A $\gst$-\emph{set} of $G$ is a strong dominating set of $G$ of minimum cardinality $\gst(G)$. If $D$ is a strong dominating set in a graph $G$, then we say that a vertex $u \in \overline{D}$ is \emph{strong dominated} by a vertex $v \in D$ if $uv\in E(G)$, and $\deg(u)\leq \deg(v)$.

In 1996, Sampathkumar and Pushpa Latha \cite{DM} introduced strong domination number and some upper bounds on this parameter presented in \cite{DM2,DM}. Similar to strong domination number, a set $D\subset V$  is a weak  dominating set of $G$, if every vertex $v\in V\setminus S$  is
adjacent to a vertex $u\in D$ such that $deg(v)\geq deg(u)$ (see \cite{Boutrig}). The minimum cardinality of a weak dominating set of $G$ is denoted by $\gw(G)$. Boutrig and  Chellali proved that for any graph $G$ of order $n\geq 3$, $\gw(G)+\frac{3}{\Delta+1}\gst(G)\leq n$. Alikhani, Ghanbari and Zaherifard \cite{sub} examined the effects on $\gamma_{st}(G)$, when $G$ is modified by the edge deletion, the edge subdivision and the edge contraction. Also they studied the strong domination number of $k$-subdivision of $G$.   
Motivated by enumerating of  the number of dominating sets of a graph and domination polynomial (see e.g. \cite{euro}), the enumeration of  the strong dominating sets for certain
graphs has studied in \cite{JAS}. 
Study of the strong domination number of graph operations are  natural and interesting subject and for join and corona products have studied in \cite{JAS}.
A domatic partition is a partition
of the vertex set into dominating sets, in other words, a partition $\pi$ = $\{V_1, V_2, . . . , V_k \}$ of $V(G)$  
such that every set $V_{i}$ is a dominating set in $G$. 
Cockayne and Hedetniemi \cite{Cockayne} introduced the domatic number of a graph $d(G)$ as the maximum order $k$ of a vertex partition. For more details on the domatic number refer to e.g., \cite{11,12,13}. 

Aram, Sheikholeslami and Volkmann in \cite{TOC} have shown that the total domatic number of a random $r$-regular graph is almost surely at most $r-1$, 
and that for $3$-regular random graphs, the total domatic number is almost surely equal to $2$. They also have given a lower bound on the total domatic number of a graph in terms of order, minimum degree and maximum degree.

Motivated by the definition of the domatic number and total domatic number, we focus on studying strong domatic number of a graph.

		A partition of $V(G)$, all of whose classes are strong dominating sets in $G$, is called a \textit{strong domatic partition} of $G$. The maximum number of classes of a strong domatic partition of $G$ is called the \textit{strong
			domatic number} of $G$ and is denoted by $\dst (G)$.

In Section 2, we compute and study the strong domatic number for certain graphs and we present different sharp bounds on $d_{st}(G)$. 
In Section 3,  we determine this parameter for all cubic graphs of order at most $10$.

\section{Results for certain graphs}

In this section, we study the strong domatic  number for certain graphs. First we state and prove the following theorem for graphs $G$ with $\delta(G)=1$.

	\begin{theorem}\label{thm:pendant}
If a graph $G$ has a pendant vertex, then $\dst (G)=1$ or $\dst (G)=2$.
	\end{theorem}

	\begin{proof}
Suppose that $u$ is a pendant vertex $u$,  $N(u)=\{v\}$ and  $P$ is  a strong domatic partition of $G$. We claim than $|P|\leq 2$. Since $\deg(u)=1$, so in any strong dominating set of $G$, say $D$, we should have either $u\in D$ or $v\in D$ or $\{u,v\}\subseteq D$. If  $\{u,v\}\subseteq D$, then by the definition of the strong dominating set and the strong domatic partition, we should have $D=V(G)$, and $P=\{D\}$. Because if we have $D'\in P$ such that $D'\neq D$, then no vertex strong dominate  $u$ which is a contradiction. The other case is $u\in D$ or $v\in D$ and not both, which in the best case gives us two strong dominating sets. Therefore we have the result. 
\qed
	\end{proof}

The following result gives bounds for the strong domatic number based on the number of vertices with maximum degree.

	\begin{theorem}\label{thm:max-degree}
Let $G$ be a graph with maximum degree $\Delta$ and $m$ be the number of vertices with degree $\Delta$. Then $1\leq \dst (G)\leq m$.
	\end{theorem}

	\begin{proof}
Since any vertex with degree $\Delta$ should be in a strong dominating set or strong dominated by another vertex with degree $\Delta$, so the maximum number of sets which are strong dominating sets and a partition of $V(G)$ is $m$, and we are done. 
\qed
	\end{proof}

	\begin{remark}\label{rem:star-complete}
Bounds in Theorem \ref{thm:max-degree} are tight. For the lower bound, it suffices to consider the star graph $K_{1,n}$. Since we only have one vertex with maximum degree, then all of vertices should be in  strong dominating set, and we have $\dst (K_{1,n})=1$. For the upper bound, it suffices to consider complete graph $K_n$. Since a single vertex is a strong dominating set, so we have  $\dst (K_{n})=n$, and we are done.
	\end{remark}

We need the following result to obtain more results:

	\begin{theorem}{\rm\cite{Cockayne}}\label{thm:domatic-min-deg}
For any graph $G$, $d(G)\leq \delta +1$, where $\delta$ is the minimum degree, and $d(G)$ is the domatic number of $G$. 
	\end{theorem}

Since in every regular graph, all vertices have the same degree, so each dominating set of a graph is a strong dominating set, too. Therefore, by Theorem \ref{thm:domatic-min-deg} we have the following result.

	\begin{corollary}\label{cor:strong-domatic-min-deg}
For any $k$-regular graph $G$, $d(G)=\dst(G)$ and $\dst(G)\leq k +1$. 
	\end{corollary}

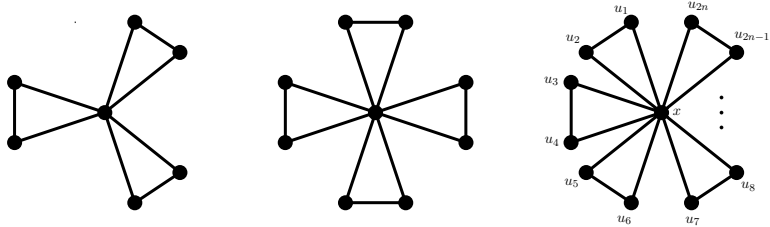
\begin{figure}
\begin{center}
\psscalebox{0.5 0.5}
{
\begin{pspicture}(0,-7.215)(20.277115,-1.245)
\psline[linecolor=black, linewidth=0.04](1.7971154,-1.815)(1.7971154,-1.815)
\psdots[linecolor=black, dotsize=0.4](8.997115,-1.815)
\psdots[linecolor=black, dotsize=0.4](10.5971155,-1.815)
\psdots[linecolor=black, dotsize=0.4](9.797115,-4.215)
\psdots[linecolor=black, dotsize=0.4](8.997115,-6.615)
\psdots[linecolor=black, dotsize=0.4](10.5971155,-6.615)
\psline[linecolor=black, linewidth=0.08](8.997115,-1.815)(10.5971155,-1.815)(8.997115,-6.615)(10.5971155,-6.615)(8.997115,-1.815)(8.997115,-1.815)
\psdots[linecolor=black, dotsize=0.4](12.197115,-3.415)
\psdots[linecolor=black, dotsize=0.4](12.197115,-5.015)
\psdots[linecolor=black, dotsize=0.4](7.397115,-3.415)
\psdots[linecolor=black, dotsize=0.4](7.397115,-5.015)
\psline[linecolor=black, linewidth=0.08](12.197115,-5.015)(7.397115,-3.415)(7.397115,-5.015)(12.197115,-3.415)(12.197115,-5.015)(12.197115,-5.015)
\psdots[linecolor=black, dotsize=0.4](0.1971154,-3.415)
\psdots[linecolor=black, dotsize=0.4](0.1971154,-5.015)
\psdots[linecolor=black, dotsize=0.4](2.5971155,-4.215)
\psline[linecolor=black, linewidth=0.08](2.5971155,-4.215)(0.1971154,-3.415)(0.1971154,-5.015)(2.5971155,-4.215)(2.5971155,-4.215)
\psdots[linecolor=black, dotsize=0.4](3.3971155,-1.815)
\psdots[linecolor=black, dotsize=0.4](4.5971155,-2.615)
\psdots[linecolor=black, dotsize=0.4](3.3971155,-6.615)
\psdots[linecolor=black, dotsize=0.4](4.5971155,-5.815)
\psline[linecolor=black, linewidth=0.08](2.5971155,-4.215)(4.5971155,-5.815)(3.3971155,-6.615)(3.3971155,-6.615)
\psline[linecolor=black, linewidth=0.08](2.5971155,-4.215)(3.3971155,-6.615)(3.3971155,-6.615)
\psline[linecolor=black, linewidth=0.08](2.5971155,-4.215)(3.3971155,-1.815)(4.5971155,-2.615)(2.5971155,-4.215)(2.5971155,-4.215)
\psdots[linecolor=black, dotsize=0.4](15.397116,-2.615)
\psdots[linecolor=black, dotsize=0.4](16.597115,-1.815)
\psdots[linecolor=black, dotsize=0.4](17.397116,-4.215)
\psdots[linecolor=black, dotsize=0.4](15.397116,-5.815)
\psdots[linecolor=black, dotsize=0.4](16.597115,-6.615)
\psdots[linecolor=black, dotsize=0.4](19.397116,-5.815)
\psdots[linecolor=black, dotsize=0.4](18.197115,-6.615)
\psdots[linecolor=black, dotsize=0.4](14.997115,-3.415)
\psdots[linecolor=black, dotsize=0.4](14.997115,-5.015)
\psdots[linecolor=black, dotsize=0.4](18.197115,-1.815)
\psdots[linecolor=black, dotsize=0.4](19.397116,-2.615)
\psdots[linecolor=black, dotsize=0.1](18.997116,-3.815)
\psdots[linecolor=black, dotsize=0.1](18.997116,-4.215)
\psdots[linecolor=black, dotsize=0.1](18.997116,-4.615)
\rput[bl](17.697115,-4.295){$x$}
\rput[bl](16.137115,-1.595){$u_1$}
\rput[bl](14.857116,-2.375){$u_2$}
\psline[linecolor=black, linewidth=0.08](17.397116,-4.215)(19.397116,-2.615)(18.197115,-1.815)(17.397116,-4.215)(16.597115,-1.815)(15.397116,-2.615)(17.397116,-4.215)(14.997115,-3.415)(14.997115,-5.015)(17.397116,-4.215)(15.397116,-5.815)(16.597115,-6.615)(17.397116,-4.215)(18.197115,-6.615)(19.397116,-5.815)(17.397116,-4.215)(17.397116,-4.215)
\rput[bl](14.277116,-3.495){$u_3$}
\rput[bl](14.297115,-5.095){$u_4$}
\rput[bl](14.817116,-6.195){$u_5$}
\rput[bl](16.217115,-7.175){$u_6$}
\rput[bl](18.037115,-7.215){$u_7$}
\rput[bl](19.517115,-6.295){$u_8$}
\rput[bl](18.097115,-1.495){$u_{2n}$}
\rput[bl](19.337116,-2.315){$u_{2n-1}$}
\end{pspicture}
}
\end{center}
\caption{Friendship graphs $F_3$, $F_4$ and $F_n$, respectively.}\label{fig:friend}
\end{figure}

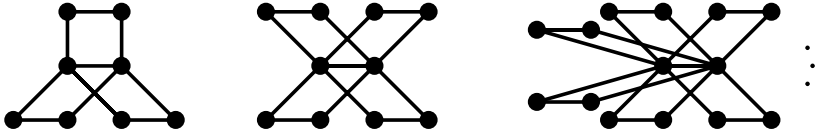
\begin{figure}
\begin{center}
\psscalebox{0.6 0.6}
{
\begin{pspicture}(0,-4.4)(17.953062,-1.605769)
\psdots[linecolor=black, dotsize=0.4](1.3971153,-3.0028846)
\psdots[linecolor=black, dotsize=0.4](2.5971153,-3.0028846)
\psdots[linecolor=black, dotsize=0.4](2.5971153,-1.8028846)
\psdots[linecolor=black, dotsize=0.4](1.3971153,-1.8028846)
\psdots[linecolor=black, dotsize=0.4](1.3971153,-4.2028847)
\psdots[linecolor=black, dotsize=0.4](0.19711533,-4.2028847)
\psdots[linecolor=black, dotsize=0.4](2.5971153,-4.2028847)
\psdots[linecolor=black, dotsize=0.4](3.7971153,-4.2028847)
\psline[linecolor=black, linewidth=0.08](2.5971153,-4.2028847)(3.7971153,-4.2028847)(2.5971153,-3.0028846)(1.3971153,-3.0028846)(2.5971153,-4.2028847)(1.3971153,-3.0028846)(1.3971153,-1.8028846)(2.5971153,-1.8028846)(2.5971153,-3.0028846)(1.3971153,-4.2028847)(0.19711533,-4.2028847)(1.3971153,-3.0028846)(1.3971153,-3.0028846)
\psline[linecolor=black, linewidth=0.08](5.7971153,-4.2028847)(6.997115,-4.2028847)(6.997115,-4.2028847)
\psdots[linecolor=black, dotsize=0.4](5.7971153,-4.2028847)
\psdots[linecolor=black, dotsize=0.4](6.997115,-4.2028847)
\psdots[linecolor=black, dotsize=0.4](6.997115,-3.0028846)
\psdots[linecolor=black, dotsize=0.4](8.197115,-3.0028846)
\psdots[linecolor=black, dotsize=0.4](8.197115,-4.2028847)
\psdots[linecolor=black, dotsize=0.4](9.397116,-4.2028847)
\psdots[linecolor=black, dotsize=0.4](8.197115,-1.8028846)
\psdots[linecolor=black, dotsize=0.4](9.397116,-1.8028846)
\psdots[linecolor=black, dotsize=0.4](6.997115,-1.8028846)
\psdots[linecolor=black, dotsize=0.4](5.7971153,-1.8028846)
\psline[linecolor=black, linewidth=0.08](5.7971153,-4.2028847)(6.997115,-3.0028846)(5.7971153,-1.8028846)(6.997115,-1.8028846)(8.197115,-3.0028846)(6.997115,-3.0028846)(8.197115,-1.8028846)(9.397116,-1.8028846)(8.197115,-3.0028846)(9.397116,-4.2028847)(8.197115,-4.2028847)(6.997115,-3.0028846)(8.197115,-3.0028846)(6.997115,-4.2028847)(6.997115,-4.2028847)
\psline[linecolor=black, linewidth=0.08](13.397116,-4.2028847)(14.5971155,-4.2028847)(14.5971155,-4.2028847)
\psdots[linecolor=black, dotsize=0.4](13.397116,-4.2028847)
\psdots[linecolor=black, dotsize=0.4](14.5971155,-4.2028847)
\psdots[linecolor=black, dotsize=0.4](14.5971155,-3.0028846)
\psdots[linecolor=black, dotsize=0.4](15.797115,-3.0028846)
\psdots[linecolor=black, dotsize=0.4](15.797115,-4.2028847)
\psdots[linecolor=black, dotsize=0.4](16.997116,-4.2028847)
\psdots[linecolor=black, dotsize=0.4](15.797115,-1.8028846)
\psdots[linecolor=black, dotsize=0.4](16.997116,-1.8028846)
\psdots[linecolor=black, dotsize=0.4](14.5971155,-1.8028846)
\psdots[linecolor=black, dotsize=0.4](13.397116,-1.8028846)
\psline[linecolor=black, linewidth=0.08](13.397116,-4.2028847)(14.5971155,-3.0028846)(13.397116,-1.8028846)(14.5971155,-1.8028846)(15.797115,-3.0028846)(14.5971155,-3.0028846)(15.797115,-1.8028846)(16.997116,-1.8028846)(15.797115,-3.0028846)(16.997116,-4.2028847)(15.797115,-4.2028847)(14.5971155,-3.0028846)(15.797115,-3.0028846)(14.5971155,-4.2028847)(14.5971155,-4.2028847)
\psdots[linecolor=black, dotsize=0.4](11.797115,-2.2028844)
\psdots[linecolor=black, dotsize=0.4](12.997115,-2.2028844)
\psdots[linecolor=black, dotsize=0.4](12.997115,-3.8028846)
\psdots[linecolor=black, dotsize=0.4](11.797115,-3.8028846)
\psline[linecolor=black, linewidth=0.08](11.797115,-2.2028844)(12.997115,-2.2028844)(11.797115,-2.2028844)(11.797115,-2.2028844)(11.797115,-2.2028844)
\psline[linecolor=black, linewidth=0.08](14.5971155,-3.0028846)(11.797115,-2.2028844)(11.797115,-2.2028844)
\psline[linecolor=black, linewidth=0.08](12.997115,-2.2028844)(15.797115,-3.0028846)(15.797115,-3.0028846)
\psline[linecolor=black, linewidth=0.08](14.5971155,-3.0028846)(11.797115,-3.8028846)(11.797115,-3.8028846)
\psline[linecolor=black, linewidth=0.08](11.797115,-3.8028846)(12.997115,-3.8028846)(15.797115,-3.0028846)(15.797115,-3.0028846)
\psdots[linecolor=black, dotsize=0.1](17.797115,-2.6028845)
\psdots[linecolor=black, dotsize=0.1](17.797115,-3.4028845)
\psdots[linecolor=black, dotsize=0.1](17.903782,-3.0028846)
\end{pspicture}
}
\end{center}
\caption{Book graph $B_3$, $B_4$ and $B_n$, respectively} \label{fig:book}
\end{figure}

The following result gives the strong domatic number of certain graphs:

	\begin{proposition}
	The following holds:
\begin{itemize}
\item[(i)]
For the path graph $P_n$, $n\geq 4$, we have $\dst(P_n)=2$.
\item[(ii)]
For the cycle graph $C_n$,
	\[
	\dst(C_n)=\left\{
	\begin{array}{ll}
	{\displaystyle
		3}&
	\quad\mbox{if $n=3k $,}\\[15pt]
	{\displaystyle
		2}&
	\quad\mbox{otherwise.}
	\end{array}
	\right.
	\]
\item[(iii)]
For the complete bipartite graph $K_{n,m}$,
	\[
	\dst(K_{n,m})=\left\{
	\begin{array}{ll}
	{\displaystyle
		1}&
	\quad\mbox{if $n<m $,}\\[15pt]
	{\displaystyle
		n}&
	\quad\mbox{if $n=m $.}
	\end{array}
	\right.
	\]
\item[(iv)]
	For the friendship graph $F_n$ (see Figure \ref{fig:friend}), $\dst(F_n)=1$.
\item[(v)]
	For the book graph $B_n$ (see Figure \ref{fig:book}), $\dst(B_n)=2$.	
\end{itemize}
	\end{proposition}

\begin{figure}
\begin{center}
\psscalebox{0.70 0.70}
{
\begin{pspicture}(0,-3.8135576)(6.4171157,-3.0464423)
\psdots[linecolor=black, dotsize=0.1](3.3971155,-3.6164422)
\psdots[linecolor=black, dotsize=0.1](3.7971156,-3.6164422)
\psdots[linecolor=black, dotsize=0.1](4.1971154,-3.6164422)
\psline[linecolor=black, linewidth=0.08](4.5971155,-3.6164422)(6.1971154,-3.6164422)(6.1971154,-3.6164422)
\psline[linecolor=black, linewidth=0.08](2.9971154,-3.6164422)(0.19711548,-3.6164422)
\psdots[linecolor=black, dotsize=0.4](0.19711548,-3.6164422)
\psdots[linecolor=black, dotsize=0.4](1.3971155,-3.6164422)
\psdots[linecolor=black, dotsize=0.4](2.5971155,-3.6164422)
\psdots[linecolor=black, dotsize=0.4](4.9971156,-3.6164422)
\psdots[linecolor=black, dotsize=0.4](6.1971154,-3.6164422)
\rput[bl](0.037115477,-3.3164423){$v_1$}
\rput[bl](1.2771155,-3.3564422){$v_2$}
\rput[bl](2.4371154,-3.3564422){$v_3$}
\rput[bl](6.0171156,-3.2964423){$v_n$}
\rput[bl](4.8371153,-3.3564422){$v_{n-1}$}
\end{pspicture}
}
\end{center}
\caption{\label{path} The path graph with $V(P_n)=\{v_1,v_2,\ldots,v_n\}$.} 
\end{figure}
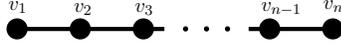

	\begin{proof}
\begin{itemize}
\item[(i)]
Suppose that $V(P_n)=\{v_1,v_2,\ldots,v_n\}$, and vertices are as in Figure \ref{path}. One can easily check that the set of vertices with even indices is a strong dominating set, and the set of vertices with odd indices is another strong dominating set. Therefore, by Theorem \ref{thm:pendant},  we have $\dst(P_n)=2$.

\item[(ii)]
Suppose that $V(C_n)=\{v_1,v_2,\ldots,v_n\}$, and vertices are in a natural order. We consider the following cases:
\begin{itemize}
\item[(a)]
$n=3k$. Let 
$$P=\Bigl\{ \{v_1,v_4,\ldots,v_{3k-2}\},\{v_2,v_5,\ldots,v_{3k-1}\},\{v_3,v_6,\ldots,v_{3k}\} \Bigl\}.$$
Clearly $P$ is a strong domatic partition of $C_{3k}$. By Corollary \ref{cor:strong-domatic-min-deg}, $\dst(C_n)\leq 3$, and therefore we are done.
\item[(b)]
$n=3k+1$. Since $\gst(C_{n})=\gamma(C_n)=\lfloor \frac{n+2}{3} \rfloor$, then $\gst(C_{3k+1})=k+1$. So a strong dominating set of $C_{3k+1}$ has at least $k+1$ vertices, which means that we can not have a strong domatic partition of $C_{3k+1}$ of size $3$. 

\item[(c)]
$n=3k+2$. By a similar argument as part (b), we have the result.

\end{itemize}

\item[(iii)]
Suppose that $V(K_{n,m})=\{v_1,v_2,\ldots,v_n,u_1,u_2,\ldots,u_m\}$, and  for $i=1,2,\ldots,n$, $N(v_i)=\{u_1,u_2,\ldots,u_m\}$. We consider the following cases:
\begin{itemize}
\item[(a)]
$n<m$. We should have all vertices in the  strong dominating set to have a partition of $V(K_{n,m})$, because no vertex can strong dominate $v_i$ for any $1\leq i\leq n$.  So $\dst(K_{n,m})=1$.
\item[(b)]
$n=m$. Let 
$$P=\Bigl\{ \{u_1,v_1\},\{u_2,v_2\},\ldots,\{u_n,v_n\} \Bigl\}.$$
Then $P$ is a strong domatic partition of $K_{n,n}$. Since set of a single vertex is not a strong dominating set of $K_{n,n}$, so we are not able to create a strong domatic partition of a bigger size. Hence $\dst(K_{n,n})=n$, and we are done.
\end{itemize}
\item[(iv)]
It is an immediate consequence  of Theorem \ref{thm:max-degree}.
\item[(v)]
Suppose that $u$ and $v$ are the vertices with maximum degree. Let $D_1=\{u\}\cup N(v)$ and $D_2=\{v\}\cup N(u)$. Clearly, $P=\{D_1,D_2\}$ is a strong domatic partition of $B_n$, and by Theorem \ref{thm:max-degree}, we have the result.
\qed
\end{itemize}
	\end{proof}

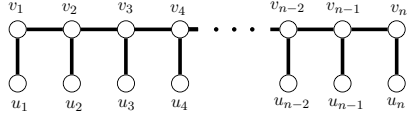
\begin{figure}
\begin{center}
\psscalebox{0.6 0.6}
{
\begin{pspicture}(0,-3.385)(8.861389,-0.955)
\psline[linecolor=black, linewidth=0.08](0.20138885,-1.585)(4.201389,-1.585)(3.8013887,-1.585)
\psline[linecolor=black, linewidth=0.08](5.8013887,-1.585)(7.4013886,-1.585)(8.601389,-1.585)(8.601389,-1.585)
\psline[linecolor=black, linewidth=0.08](0.20138885,-1.585)(0.20138885,-2.785)(0.20138885,-2.785)
\psline[linecolor=black, linewidth=0.08](1.4013889,-1.585)(1.4013889,-2.785)(1.4013889,-2.785)
\psline[linecolor=black, linewidth=0.08](2.601389,-1.585)(2.601389,-2.785)(2.601389,-2.785)
\psline[linecolor=black, linewidth=0.08](3.8013887,-1.585)(3.8013887,-2.785)(3.8013887,-2.785)
\psline[linecolor=black, linewidth=0.08](6.201389,-1.585)(6.201389,-2.785)(6.201389,-2.785)
\psline[linecolor=black, linewidth=0.08](7.4013886,-1.585)(7.4013886,-2.785)(7.4013886,-2.785)
\psline[linecolor=black, linewidth=0.08](8.601389,-1.585)(8.601389,-2.785)(8.601389,-2.785)
\psdots[linecolor=black, fillstyle=solid, dotstyle=o, dotsize=0.4, fillcolor=white](0.20138885,-1.585)
\psdots[linecolor=black, fillstyle=solid, dotstyle=o, dotsize=0.4, fillcolor=white](1.4013889,-1.585)
\psdots[linecolor=black, fillstyle=solid, dotstyle=o, dotsize=0.4, fillcolor=white](2.601389,-1.585)
\psdots[linecolor=black, fillstyle=solid, dotstyle=o, dotsize=0.4, fillcolor=white](3.8013887,-1.585)
\psdots[linecolor=black, fillstyle=solid, dotstyle=o, dotsize=0.4, fillcolor=white](6.201389,-1.585)
\psdots[linecolor=black, fillstyle=solid, dotstyle=o, dotsize=0.4, fillcolor=white](7.4013886,-1.585)
\psdots[linecolor=black, fillstyle=solid, dotstyle=o, dotsize=0.4, fillcolor=white](8.601389,-1.585)
\psdots[linecolor=black, fillstyle=solid, dotstyle=o, dotsize=0.4, fillcolor=white](0.20138885,-2.785)
\psdots[linecolor=black, fillstyle=solid, dotstyle=o, dotsize=0.4, fillcolor=white](1.4013889,-2.785)
\psdots[linecolor=black, fillstyle=solid, dotstyle=o, dotsize=0.4, fillcolor=white](2.601389,-2.785)
\psdots[linecolor=black, fillstyle=solid, dotstyle=o, dotsize=0.4, fillcolor=white](3.8013887,-2.785)
\psdots[linecolor=black, fillstyle=solid, dotstyle=o, dotsize=0.4, fillcolor=white](6.201389,-2.785)
\psdots[linecolor=black, fillstyle=solid, dotstyle=o, dotsize=0.4, fillcolor=white](7.4013886,-2.785)
\psdots[linecolor=black, fillstyle=solid, dotstyle=o, dotsize=0.4, fillcolor=white](8.601389,-2.785)
\psdots[linecolor=black, dotsize=0.1](4.601389,-1.585)
\psdots[linecolor=black, dotsize=0.1](5.001389,-1.585)
\psdots[linecolor=black, dotsize=0.1](5.4013886,-1.585)
\rput[bl](0.0013888549,-1.225){$v_1$}
\rput[bl](1.1813889,-1.245){$v_2$}
\rput[bl](2.4213889,-1.225){$v_3$}
\rput[bl](3.581389,-1.265){$v_4$}
\rput[bl](5.8013887,-1.225){$v_{n-2}$}
\rput[bl](7.041389,-1.245){$v_{n-1}$}
\rput[bl](8.461389,-1.265){$v_n$}
\rput[bl](0.061388854,-3.385){$u_1$}
\rput[bl](1.2613889,-3.385){$u_2$}
\rput[bl](2.4213889,-3.365){$u_3$}
\rput[bl](3.601389,-3.365){$u_4$}
\rput[bl](8.381389,-3.325){$u_n$}
\rput[bl](5.8813887,-3.345){$u_{n-2}$}
\rput[bl](7.081389,-3.365){$u_{n-1}$}
\end{pspicture}
}
\end{center}
\caption{$P_n \circ K_1$.} \label{fig:PnoK1}
\end{figure}

	The corona product of two graphs $F$ and $H$,  denoted by $F\circ H$,  is defined as the graph obtained by taking one copy of $F$ and $|V(F)|$ copies of $H$ and joining the $i$-th vertex of $F$ to every vertex in the $i$-th copy of $H$. 
	The following theorem gives the strong domatic number of corona of path and cycle graph with $K_1$. 
	
	\begin{theorem}
	The following holds:
\begin{itemize}
\item[(i)] For any $n\geq 2$, 
$\dst(P_n \circ K_1)=2$.

\item[(ii)] For any $n\geq 3$, 
$\dst(C_n \circ K_1)=2$.

\end{itemize}	
	\end{theorem}

	\begin{proof}
\begin{itemize}
\item[(i)]
Consider graph $P_n \circ K_1$, as we see in Figure \ref{fig:PnoK1}. Let 
$$P=\Bigl\{ \{v_1,u_2,v_3,u_4,\ldots,v_{2t+1},u_{2t+2},\ldots\},\{u_1,v_2,u_3,v_4,\ldots,u_{2t+1},v_{2t+2},\ldots\} \Bigl\}.$$
It is easy that   $P$ is a strong domatic partition of $P_n \circ K_1$. Therefore by Theorem \ref{thm:pendant}, we have the result.
\item[(ii)]
By a similar argument as Part (i), we have the result.
\qed
\end{itemize}
	\end{proof}

The following theorem gives bounds for the strong domatic number of corona of two graphs. 

\begin{theorem}\label{thm:corona}
	Let $G$ and $H$ be two graphs. We have 
	$$1 \leq \dst(G\circ H)\leq \dst(G).$$	
\end{theorem}

\begin{proof}
	Note that the set of a set including all vertices is a strong domatic partition of $G\circ H$, and we have nothing to prove for the lower bound. Now, we consider the upper bound and prove it. Suppose that $V(G)=\{v_1,v_2,\ldots,v_n\}$, and for the copy of $H$ related to vertex $v_i$, for $i=1,2,\ldots,n$, $V(H_{v_i})=\{u_{i_1},u_{i_2},\ldots,u_{i_m}\}$. By the definition of $G\circ H$ it is clear that $\deg (u_{i_j})< \deg (v_i)$, for all $j=1,2,\ldots,m$. So, there is no vertex in $V(H_{v_i})$ such that strong dominate $v_i$, for $i=1,2,\ldots,n$. Therefore, in the best case, we can find $\dst(G)$ sets to have a strong domatic partition of $G\circ H$, and we are done. 
	\qed
\end{proof}

\begin{remark}
	Bounds in Theorem \ref{thm:corona} are tight. For the lower bound, it suffices to consider $G=\overline{K_n}$ and $H=\overline{K_m}$. Then $G\circ H$ is the union of $n$ star graphs $K_{1,m}$. As shown in Remark \ref{rem:star-complete}, we have $\dst(G\circ H)=1$. For the upper bound let $G=H=K_n$. As shown in Remark \ref{rem:star-complete}, $\dst(G)=n$. Now, we present a strong domatic partition of $G\circ H$ of size $n$. Suppose that $V(G)=\{v_1,v_2,\ldots,v_n\}$, and for the copy of $H=K_n$ related to vertex $v_i$, for $i=1,2,\ldots,n$, $V(H_{v_i})=\{u_{i_1},u_{i_2},\ldots,u_{i_n}\}$. Let 
	$$A_i=\{v_i, u_{1_i}, u_{2_i}, u_{3_i}, \ldots,u_{n_i} \},$$
	for $i=1,2,\ldots,n$. Then,  
	$$P=\{A_1,A_2,A_3,\ldots,A_n\}$$
	is a strong domatic partition of $G\circ H=K_n\circ K_n$, and we have the result.  
\end{remark}

	\section{Computing  $d_{st}(G)$ for cubic graphs of order at most $10$}
	The class of cubic graphs
	is especially interesting for mathematical applications, because for various important open problems in
	graph theory, cubic graphs are the smallest or simplest possible potential counterexamples, and so this creates motivation to study  strong domatic number  for the cubic graphs of order at most $10$.
	
	Alikhani and Peng have studied the domination polynomials (which is the generating function for the number of dominating sets of a graph) of cubic graphs of order $10$ in \cite{1}. As a 	consequence, they have shown that the Petersen graph is determined uniquely by its domination polynomial. Ghanbari has studied  the Sombor characteristic polynomial and Sombor energy of these graphs in \cite{Energy}, and has shown that the  Petersen graph is not determined uniquely by its Sombor energy, but it has the maximum Sombor energy among others.

	First, we  determine the strong domatic number of the cubic graphs of order $6$. There are exactly two  cubic graphs of order $6$ which are denoted by $G_{1}$ and $G_{2}$ in Figure \ref{fig:Cubic6}.
	
\begin{figure}
\begin{center}
\psscalebox{0.7 0.7}
{
\begin{pspicture}(0,-4.575)(8.82,-0.085)
\psline[linecolor=black, linewidth=0.04](0.4,-1.355)(1.6,-0.555)(2.8,-1.355)(2.8,-2.555)(1.6,-3.355)(0.4,-2.555)(0.4,-1.355)(0.4,-1.355)
\psline[linecolor=black, linewidth=0.04](1.6,-0.555)(1.6,-3.355)
\psline[linecolor=black, linewidth=0.04](0.4,-1.355)(2.8,-1.355)
\psline[linecolor=black, linewidth=0.04](0.4,-2.555)(2.8,-2.555)
\rput[bl](1.52,-0.355){1}
\rput[bl](3.0,-1.475){2}
\rput[bl](2.98,-2.695){3}
\rput[bl](0.04,-2.735){5}
\rput[bl](0.0,-1.495){6}
\psdots[linecolor=black, dotsize=0.2](1.6,-0.555)
\psdots[linecolor=black, dotsize=0.2](0.4,-1.355)
\psdots[linecolor=black, dotsize=0.2](0.4,-2.555)
\psdots[linecolor=black, dotsize=0.2](2.8,-1.355)
\psdots[linecolor=black, dotsize=0.2](2.8,-2.555)
\psdots[linecolor=black, dotsize=0.2](1.6,-3.355)
\rput[bl](1.48,-3.855){4}
\psline[linecolor=black, linewidth=0.04](6.0,-1.355)(7.2,-0.555)(8.4,-1.355)(8.4,-2.555)(7.2,-3.355)(6.0,-2.555)(6.0,-1.355)(6.0,-1.355)
\psline[linecolor=black, linewidth=0.04](7.2,-0.555)(7.2,-3.355)
\rput[bl](7.12,-0.355){1}
\rput[bl](8.6,-1.475){2}
\rput[bl](8.58,-2.695){3}
\rput[bl](5.64,-2.735){5}
\rput[bl](5.6,-1.495){6}
\psdots[linecolor=black, dotsize=0.2](7.2,-0.555)
\psdots[linecolor=black, dotsize=0.2](6.0,-1.355)
\psdots[linecolor=black, dotsize=0.2](6.0,-2.555)
\psdots[linecolor=black, dotsize=0.2](8.4,-1.355)
\psdots[linecolor=black, dotsize=0.2](8.4,-2.555)
\psdots[linecolor=black, dotsize=0.2](7.2,-3.355)
\rput[bl](7.08,-3.855){4}
\psline[linecolor=black, linewidth=0.04](6.0,-1.355)(8.4,-2.555)(8.4,-2.555)
\psline[linecolor=black, linewidth=0.04](8.4,-1.355)(6.0,-2.555)(6.0,-2.555)
\rput[bl](1.38,-4.575){\large{$G_1$}}
\rput[bl](7.02,-4.555){\large{$G_2$}}
\end{pspicture}
}
\end{center}
\caption{\label{fig:Cubic6} Cubic graphs of order $6$.}
	\end{figure}
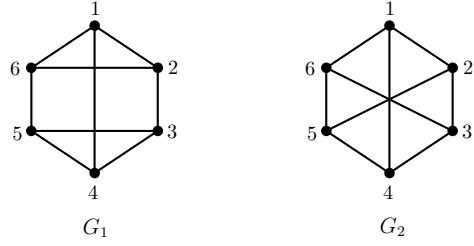

	\begin{theorem}\label{thm:cubic6}
		The strong domatic number of the cubic graphs $G_{1}$ and $G_{2}$  (Figure \ref{fig:Cubic6}) of order $6$ is $3$.  
	\end{theorem}

	\begin{proof} 
It is clear that a single vertex cannot strong dominate all other vertices. So, we need at least two vertices in any strong dominating sets of $G_1$ and $G_2$. We see that   	
$$P = \Bigl\{ \{1,4 \},\{2,3 \},\{5,6 \} \Bigl\}$$
is a strong domatic partition of $G_1$ and also $G_2$. Therefore we have the result.
		\qed
	\end{proof}

	\medskip
	
	Now, we compute the  strong domatic number of cubic graphs of order $8$. There are exactly $6$ cubic graphs of order $8$ which is denoted by  $G_{1},G_{2},...,G_{6}$ in Figure \ref{fig:Cubic8}.  	The following theorem gives the strong domatic numbers of cubic graphs of order $8$: 
	
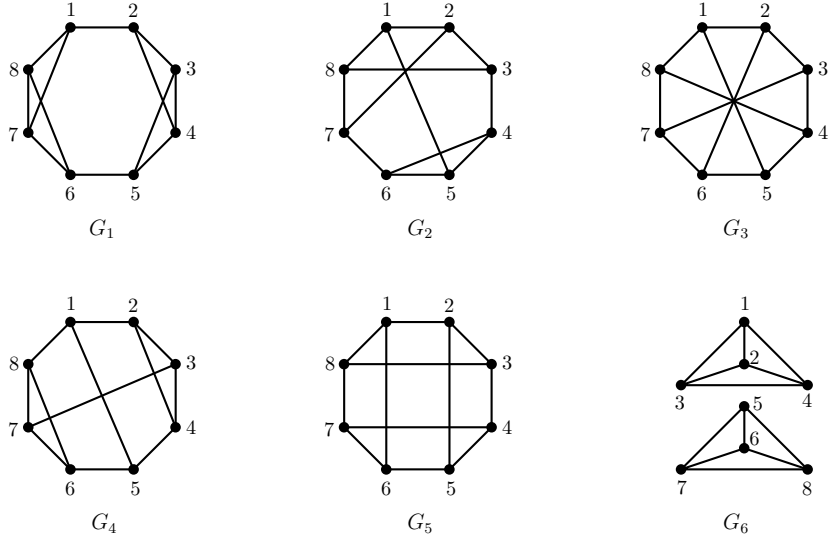
\begin{figure}
\begin{center}
\psscalebox{0.7 0.7}
{
\begin{pspicture}(0,-7.7901173)(15.509375,2.3501172)
\rput[bl](1.1,2.0901172){1}
\rput[bl](2.28,2.0701172){2}
\rput[bl](3.38,0.9701172){3}
\rput[bl](2.32,-1.4098828){5}
\rput[bl](1.08,-1.4098828){6}
\rput[bl](3.38,-0.24988282){4}
\rput[bl](7.58,-2.1298828){\large{$G_2$}}
\psdots[linecolor=black, dotsize=0.2](1.18,1.8701172)
\psdots[linecolor=black, dotsize=0.2](2.38,1.8701172)
\psdots[linecolor=black, dotsize=0.2](2.38,-0.9298828)
\psdots[linecolor=black, dotsize=0.2](1.18,-0.9298828)
\rput[bl](0.0,-0.2898828){7}
\rput[bl](0.02,0.9501172){8}
\rput[bl](7.1,2.0901172){1}
\rput[bl](8.28,2.0701172){2}
\rput[bl](9.38,0.9701172){3}
\rput[bl](8.32,-1.4098828){5}
\rput[bl](7.08,-1.4098828){6}
\rput[bl](9.38,-0.24988282){4}
\psdots[linecolor=black, dotsize=0.2](7.18,1.8701172)
\psdots[linecolor=black, dotsize=0.2](8.38,1.8701172)
\psdots[linecolor=black, dotsize=0.2](8.38,-0.9298828)
\psdots[linecolor=black, dotsize=0.2](7.18,-0.9298828)
\rput[bl](6.0,-0.2898828){7}
\rput[bl](6.02,0.9501172){8}
\rput[bl](13.9,-3.509883){1}
\rput[bl](14.08,-4.529883){2}
\rput[bl](12.66,-5.389883){3}
\rput[bl](14.14,-5.409883){5}
\rput[bl](14.08,-6.069883){6}
\rput[bl](15.08,-5.3498826){4}
\rput[bl](12.7,-6.989883){7}
\rput[bl](15.08,-7.009883){8}
\rput[bl](1.54,-2.1298828){\large{$G_1$}}
\rput[bl](13.58,-2.1298828){\large{$G_3$}}
\rput[bl](1.58,-7.7298827){\large{$G_4$}}
\rput[bl](7.58,-7.7298827){\large{$G_5$}}
\rput[bl](13.58,-7.7298827){\large{$G_6$}}
\psline[linecolor=black, linewidth=0.04](1.18,-3.7298827)(1.18,-3.7298827)
\psdots[linecolor=black, dotsize=0.2](12.78,-6.529883)
\psdots[linecolor=black, dotsize=0.2](15.18,-6.529883)
\psdots[linecolor=black, dotsize=0.2](13.98,-6.129883)
\psdots[linecolor=black, dotsize=0.2](13.98,-5.3298826)
\psline[linecolor=black, linewidth=0.04](13.98,-5.3298826)(12.78,-6.529883)(15.18,-6.529883)(13.98,-5.3298826)(13.98,-6.129883)(12.78,-6.529883)(12.78,-6.529883)
\psline[linecolor=black, linewidth=0.04](13.98,-6.129883)(15.18,-6.529883)(15.18,-6.529883)
\psdots[linecolor=black, dotsize=0.2](12.78,-4.929883)
\psdots[linecolor=black, dotsize=0.2](15.18,-4.929883)
\psdots[linecolor=black, dotsize=0.2](13.98,-4.529883)
\psdots[linecolor=black, dotsize=0.2](13.98,-3.7298827)
\psline[linecolor=black, linewidth=0.04](13.98,-3.7298827)(12.78,-4.929883)(15.18,-4.929883)(13.98,-3.7298827)(13.98,-4.529883)(12.78,-4.929883)(12.78,-4.929883)
\psline[linecolor=black, linewidth=0.04](13.98,-4.529883)(15.18,-4.929883)(15.18,-4.929883)
\psdots[linecolor=black, dotsize=0.2](3.18,1.0701172)
\psdots[linecolor=black, dotsize=0.2](3.18,-0.12988281)
\psdots[linecolor=black, dotsize=0.2](0.38,1.0701172)
\psdots[linecolor=black, dotsize=0.2](0.38,-0.12988281)
\psline[linecolor=black, linewidth=0.04](2.38,1.8701172)(3.18,1.0701172)(3.18,-0.12988281)(2.38,-0.9298828)(2.38,-0.9298828)
\psline[linecolor=black, linewidth=0.04](1.18,1.8701172)(0.38,1.0701172)(0.38,-0.12988281)(1.18,-0.9298828)(1.18,-0.9298828)
\psline[linecolor=black, linewidth=0.04](1.18,1.8701172)(0.38,-0.12988281)(0.38,-0.12988281)
\psline[linecolor=black, linewidth=0.04](0.38,1.0701172)(1.18,-0.9298828)(1.18,-0.9298828)
\psline[linecolor=black, linewidth=0.04](2.38,1.8701172)(3.18,-0.12988281)(3.18,-0.12988281)
\psline[linecolor=black, linewidth=0.04](3.18,1.0701172)(2.38,-0.9298828)(2.38,-0.9298828)
\psline[linecolor=black, linewidth=0.04](1.18,1.8701172)(2.38,1.8701172)(2.38,1.8701172)
\psline[linecolor=black, linewidth=0.04](1.18,-0.9298828)(2.38,-0.9298828)(2.38,-0.9298828)
\psdots[linecolor=black, dotsize=0.2](6.38,1.0701172)
\psdots[linecolor=black, dotsize=0.2](6.38,-0.12988281)
\psdots[linecolor=black, dotsize=0.2](9.18,1.0701172)
\psdots[linecolor=black, dotsize=0.2](9.18,-0.12988281)
\psline[linecolor=black, linewidth=0.04](7.18,1.8701172)(8.38,1.8701172)(9.18,1.0701172)(9.18,-0.12988281)(8.38,-0.9298828)(7.18,-0.9298828)(6.38,-0.12988281)(6.38,1.0701172)(7.18,1.8701172)(7.18,1.8701172)(7.18,1.8701172)
\rput[bl](7.1,-3.509883){1}
\rput[bl](8.28,-3.529883){2}
\rput[bl](9.38,-4.629883){3}
\rput[bl](8.32,-7.009883){5}
\rput[bl](7.08,-7.009883){6}
\rput[bl](9.38,-5.8498826){4}
\psdots[linecolor=black, dotsize=0.2](7.18,-3.7298827)
\psdots[linecolor=black, dotsize=0.2](8.38,-3.7298827)
\psdots[linecolor=black, dotsize=0.2](8.38,-6.529883)
\psdots[linecolor=black, dotsize=0.2](7.18,-6.529883)
\rput[bl](6.0,-5.889883){7}
\rput[bl](6.02,-4.649883){8}
\psdots[linecolor=black, dotsize=0.2](6.38,-4.529883)
\psdots[linecolor=black, dotsize=0.2](6.38,-5.7298827)
\psdots[linecolor=black, dotsize=0.2](9.18,-4.529883)
\psdots[linecolor=black, dotsize=0.2](9.18,-5.7298827)
\psline[linecolor=black, linewidth=0.04](7.18,-3.7298827)(8.38,-3.7298827)(9.18,-4.529883)(9.18,-5.7298827)(8.38,-6.529883)(7.18,-6.529883)(6.38,-5.7298827)(6.38,-4.529883)(7.18,-3.7298827)(7.18,-3.7298827)(7.18,-3.7298827)
\psline[linecolor=black, linewidth=0.04](7.18,1.8701172)(8.38,-0.9298828)(8.38,-0.9298828)
\psline[linecolor=black, linewidth=0.04](8.38,1.8701172)(6.38,-0.12988281)(6.38,-0.12988281)
\psline[linecolor=black, linewidth=0.04](9.18,1.0701172)(6.38,1.0701172)(6.38,1.0701172)
\psline[linecolor=black, linewidth=0.04](9.18,-0.12988281)(7.18,-0.9298828)(7.18,-0.9298828)
\rput[bl](13.1,2.0901172){1}
\rput[bl](14.28,2.0701172){2}
\rput[bl](15.38,0.9701172){3}
\rput[bl](14.32,-1.4098828){5}
\rput[bl](13.08,-1.4098828){6}
\rput[bl](15.38,-0.24988282){4}
\psdots[linecolor=black, dotsize=0.2](13.18,1.8701172)
\psdots[linecolor=black, dotsize=0.2](14.38,1.8701172)
\psdots[linecolor=black, dotsize=0.2](14.38,-0.9298828)
\psdots[linecolor=black, dotsize=0.2](13.18,-0.9298828)
\rput[bl](12.0,-0.2898828){7}
\rput[bl](12.02,0.9501172){8}
\psdots[linecolor=black, dotsize=0.2](12.38,1.0701172)
\psdots[linecolor=black, dotsize=0.2](12.38,-0.12988281)
\psdots[linecolor=black, dotsize=0.2](15.18,1.0701172)
\psdots[linecolor=black, dotsize=0.2](15.18,-0.12988281)
\psline[linecolor=black, linewidth=0.04](13.18,1.8701172)(14.38,1.8701172)(15.18,1.0701172)(15.18,-0.12988281)(14.38,-0.9298828)(13.18,-0.9298828)(12.38,-0.12988281)(12.38,1.0701172)(13.18,1.8701172)(13.18,1.8701172)(13.18,1.8701172)
\psline[linecolor=black, linewidth=0.04](13.18,1.8701172)(14.38,-0.9298828)(14.38,-0.9298828)
\psline[linecolor=black, linewidth=0.04](14.38,1.8701172)(13.18,-0.9298828)(13.18,-0.9298828)
\psline[linecolor=black, linewidth=0.04](15.18,1.0701172)(12.38,-0.12988281)(12.38,-0.12988281)
\psline[linecolor=black, linewidth=0.04](12.38,1.0701172)(15.18,-0.12988281)(15.18,-0.12988281)
\rput[bl](1.1,-3.509883){1}
\rput[bl](2.28,-3.529883){2}
\rput[bl](3.38,-4.629883){3}
\rput[bl](2.32,-7.009883){5}
\rput[bl](1.08,-7.009883){6}
\rput[bl](3.38,-5.8498826){4}
\psdots[linecolor=black, dotsize=0.2](1.18,-3.7298827)
\psdots[linecolor=black, dotsize=0.2](2.38,-3.7298827)
\psdots[linecolor=black, dotsize=0.2](2.38,-6.529883)
\psdots[linecolor=black, dotsize=0.2](1.18,-6.529883)
\rput[bl](0.0,-5.889883){7}
\rput[bl](0.02,-4.649883){8}
\psdots[linecolor=black, dotsize=0.2](0.38,-4.529883)
\psdots[linecolor=black, dotsize=0.2](0.38,-5.7298827)
\psdots[linecolor=black, dotsize=0.2](3.18,-4.529883)
\psdots[linecolor=black, dotsize=0.2](3.18,-5.7298827)
\psline[linecolor=black, linewidth=0.04](1.18,-3.7298827)(2.38,-3.7298827)(3.18,-4.529883)(3.18,-5.7298827)(2.38,-6.529883)(1.18,-6.529883)(0.38,-5.7298827)(0.38,-4.529883)(1.18,-3.7298827)(1.18,-3.7298827)(1.18,-3.7298827)
\psline[linecolor=black, linewidth=0.04](1.18,-3.7298827)(2.38,-6.529883)(2.38,-6.529883)
\psline[linecolor=black, linewidth=0.04](2.38,-3.7298827)(3.18,-5.7298827)(3.18,-5.7298827)
\psline[linecolor=black, linewidth=0.04](0.38,-4.529883)(1.18,-6.529883)(1.18,-6.529883)
\psline[linecolor=black, linewidth=0.04](0.38,-5.7298827)(3.18,-4.529883)
\psline[linecolor=black, linewidth=0.04](7.18,-3.7298827)(7.18,-6.529883)(7.18,-6.529883)
\psline[linecolor=black, linewidth=0.04](8.38,-3.7298827)(8.38,-6.529883)(8.38,-6.529883)
\psline[linecolor=black, linewidth=0.04](9.18,-4.529883)(6.38,-4.529883)
\psline[linecolor=black, linewidth=0.04](6.38,-5.7298827)(9.18,-5.7298827)
\end{pspicture}
}
\end{center}
\caption{\label{fig:Cubic8} Cubic graphs of order $8$.}
	\end{figure}

	\begin{theorem}\label{thm:cubic8}
		For the cubic graphs $G_{1},G_{2},...,G_{6}$ of order $8$ (Figure \ref{fig:Cubic8}) we have: 	
\begin{enumerate} 
\item[(i)] 
$\dst(G_1)=\dst(G_5)=\dst(G_6)=4.$			
\item[(ii)] 
$\dst(G_2)=\dst(G_3)=2.$
\item[(iii)] 
$\dst(G_4)=3$			
\end{enumerate} 
	\end{theorem}

	\begin{proof} 
\begin{enumerate} 
\item[(i)] 
By Theorem \ref{thm:max-degree}, for a cubic graph $G$ of order $8$ we have $\dst(G)\leq 4$. Now we present the strong domatic partition of size $4$ for $G_1$, $G_5$ and $G_6$. Consider the following sets:
\begin{align*}
P_{1} &= \Bigl\{ \{1,5\},\{2,6\},\{3,7\},\{4,8 \} \Bigl\}, &
P_{5} &= \Bigl\{ \{1,4\},\{2,7\},\{3,6\},\{5,8 \} \Bigl\},\\
P_{6} &= \Bigl\{ \{1,5\},\{2,6\},\{3,7\},\{4,8 \} \Bigl\}.
\end{align*}
Observe that  $P_i$ is 	a strong domatic partition of $G_i$, for $i=1,5,6$ and so we have the result.
\item[(ii)] 
Suppose that $D$ is a strong dominating set of $G_2$. We show that $|D|\geq 3$. If we have two adjacent vertices in $D$, then  at least one vertex is not strong dominate by them. So we consider other cases. If $1\in D$, then it strong dominate  $2,5,7$, and we need at least two vertices among $3,4,6,8$ to be in $D$. If $2\in D$, then it strong dominate  $1,3,8$, and we need at least two vertices among $4,5,6,7$ to be in $D$. If $3\in D$, then it strong dominate  $2,4,8$, and we need at least two vertices among $1,5,6,7$ to be in $D$. If $4\in D$, then it strong dominate  $3,5,6$, and we need at least two vertices among $1,2,7,8$ to be in $D$. If $5\in D$, then it  strong dominate  $1,4,6$, and we need at least two vertices among $2,3,7,8$ to be in $D$. If $6\in D$, then it strong dominate  $4,5,7$, and we need at least two vertices among $1,2,3,8$ to be in $D$. If $7\in D$, then it strong dominate  $2,6,8$, and we need at least two vertices among $1,3,4,5$ to be in $D$. And finally if $8\in D$, then it  strong dominate  $1,3,7$, and we need at least two vertices among $2,4,5,6$ to be in $D$. So $|D|\geq 3$. Suppose that $P$ is a strong domatic partition of $G_2$ of the biggest size. By our argument $|P|$ cannot be $3$ or $4$, because then we need a strong dominating set of size $2$. So $|P|\leq 2$. It is clear that 
$$P_2 = \Bigl\{ \{1,3,5,7\},\{2,4,6,8\} \Bigl\}$$
is a strong domatic partition of $G_2$, and we are done.  By a similar argument we have $\dst(G_3)=2$.
\item[(iii)]
For $G_3$ it is possible to have strong dominating sets of size $2$ which are 
$\{2,6\}$ and $\{4,8\}$. Now suppose that $D$ is a strong dominating set of $G_5$ and $1\in D$. By a similar argument as part (ii) we conclude that $|D|\geq 3$. Now suppose that $P$ is a strong domatic partition of $G_5$ of the biggest size. By our argument $|P|$ cannot be $4$, because then we need that all of strong dominating sets be of size $2$. So $|P|\leq 3$. It is clear that 
$$P_5 = \Bigl\{ \{2,6\},\{4,8\},\{1,3,5,7\} \Bigl\}$$
is a strong domatic partition of $G_2$, and we are done.
\qed	
\end{enumerate} 
	\end{proof}

  One of the famous cubic graphs is the Petersen
graph which is a symmetric non-planar $3$-regular graph of order $10$.  
There are exactly twenty one $3$-regular graphs of order $10$ \cite{1}. 
Now, we study the strong domatic number of cubic graphs of order $10$.

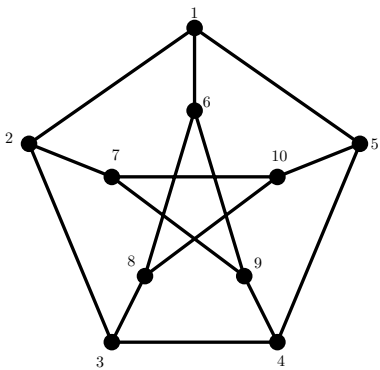
\begin{figure}
	\begin{center}
		\psscalebox{0.55 0.55}
{
\begin{pspicture}(0,-5.845)(9.06,2.865)
\psdots[linecolor=black, dotsize=0.4](4.58,2.355)
\psdots[linecolor=black, dotsize=0.4](4.58,0.355)
\psdots[linecolor=black, dotsize=0.4](2.58,-1.245)
\psdots[linecolor=black, dotsize=0.4](3.38,-3.645)
\psdots[linecolor=black, dotsize=0.4](5.78,-3.645)
\psdots[linecolor=black, dotsize=0.4](6.58,-1.245)
\psline[linecolor=black, linewidth=0.08](4.58,0.355)(3.38,-3.645)(6.58,-1.245)(2.58,-1.245)(5.78,-3.645)(4.58,0.355)(4.58,0.355)
\psdots[linecolor=black, dotsize=0.4](8.58,-0.445)
\psdots[linecolor=black, dotsize=0.4](0.58,-0.445)
\psdots[linecolor=black, dotsize=0.4](6.58,-5.245)
\psdots[linecolor=black, dotsize=0.4](2.58,-5.245)
\psline[linecolor=black, linewidth=0.08](4.58,2.355)(8.58,-0.445)(6.58,-5.245)(2.58,-5.245)(0.58,-0.445)(4.58,2.355)(4.58,0.355)(4.58,0.355)
\psline[linecolor=black, linewidth=0.08](6.58,-1.245)(8.58,-0.445)(8.58,-0.445)
\psline[linecolor=black, linewidth=0.08](5.78,-3.645)(6.58,-5.245)(6.58,-5.245)
\psline[linecolor=black, linewidth=0.08](3.38,-3.645)(2.58,-5.245)(2.58,-5.245)
\psline[linecolor=black, linewidth=0.08](2.58,-1.245)(0.58,-0.445)(0.58,-0.445)
\rput[bl](4.48,2.595){1}
\rput[bl](4.78,0.435){6}
\rput[bl](0.0,-0.425){2}
\rput[bl](2.2,-5.845){3}
\rput[bl](6.58,-5.825){4}
\rput[bl](8.84,-0.565){5}
\rput[bl](2.58,-0.845){7}
\rput[bl](2.96,-3.405){8}
\rput[bl](6.02,-3.445){9}
\rput[bl](6.44,-0.865){10}
\end{pspicture}
}
	\end{center}
	\caption{Petersen graph $P$. } \label{fig:petersen}
\end{figure}

First we state and prove the following theorem for the Petersen graph.

	\begin{theorem}\label{thm:petersen}
For the Petersen graph, $\dst(P)=2$.
	\end{theorem}

	\begin{proof}
 Suppose that $S$ is a strong dominating set of $P$. Since each vertex in $S$  strong dominate  at most $3$ other vertices,  we need to have $|S|\geq 3$. Consider Figure \ref{fig:petersen}. Note that no subset of size three of $A=\{1,2,3,4,5\}$ or $B=\{6,7,8,9,10\}$ is a strong dominating set of $P$. So, we need at least one element of $A$, and at least one element of $B$. Now, we claim that if we have a strong dominating set of size $3$, then it is not possible to have a strong domatic partition of $P$ of size $3$. We consider vertex $1\in A$. One can easily check that the only possible strong dominating sets of $P$ of size three, which contain  $1$, are the following:
\begin{align*}
S_{1} &=  \{1,3,7\}, &
S_{2} &=  \{1,4,10\}, &
S_{3} &=  \{1,8,9\}.
\end{align*} 
Since all of the elements of $S_1$ strong dominate  $2$ and  $N(2)=S_1$, so  clearly it is not possible to have a strong domatic partition of $P$ of size $3$. By the same reason, since $N(5)=S_2$ and $N(6)=S_3$, so it is not possible to have a strong domatic partition of $P$ of size $3$ including $1$. So we need to have $1$ in a strong dominating set of bigger size. Since Petersen graph is a symmetric graph, this argument holds for all vertices. So, if we have a strong dominating set of size $3$, then it is not possible to have a strong domatic partition of $P$ of size $3$, as we claimed. Since we have only  $10$ vertices, it is not possible to have a strong domatic partition of $P$ of size three and it has at least four elements. So $\dst(P)\leq 2$. Clearly, 
$P=\{A,B\}$ is a strong domatic partition of $P$, and therefore we have the result. 
\qed
	\end{proof}

\begin{figure}[!h]
\begin{center}
\psscalebox{0.69 0.69}
{
\begin{pspicture}(0,-10.145)(21.18,9.525)
\rput[bl](1.48,9.255){1}
\rput[bl](2.26,9.235){2}
\rput[bl](2.96,8.135){3}
\rput[bl](2.3,6.155){5}
\rput[bl](1.46,6.155){6}
\rput[bl](2.96,7.315){4}
\rput[bl](0.66,6.175){7}
\rput[bl](0.0,7.315){8}
\rput[bl](1.4,5.455){\large{$G_1$}}
\rput[bl](0.0,8.095){9}
\rput[bl](0.52,9.255){10}
\rput[bl](5.0,5.455){\large{$G_2$}}
\rput[bl](8.6,5.455){\large{$G_3$}}
\rput[bl](12.2,5.455){\large{$G_4$}}
\rput[bl](15.8,5.455){\large{$G_5$}}
\rput[bl](19.4,5.455){\large{$G_6$}}
\psdots[linecolor=black, dotsize=0.2](0.76,9.035)
\psdots[linecolor=black, dotsize=0.2](1.56,9.035)
\psdots[linecolor=black, dotsize=0.2](2.36,9.035)
\psdots[linecolor=black, dotsize=0.2](2.76,8.235)
\psdots[linecolor=black, dotsize=0.2](2.76,7.435)
\psdots[linecolor=black, dotsize=0.2](2.36,6.635)
\psdots[linecolor=black, dotsize=0.2](1.56,6.635)
\psdots[linecolor=black, dotsize=0.2](0.76,6.635)
\psdots[linecolor=black, dotsize=0.2](0.36,7.435)
\psdots[linecolor=black, dotsize=0.2](0.36,8.235)
\psline[linecolor=black, linewidth=0.04](0.76,9.035)(0.36,8.235)(0.36,7.435)(0.76,6.635)(0.76,6.635)
\psline[linecolor=black, linewidth=0.04](2.36,9.035)(2.76,8.235)(2.76,7.435)(2.36,6.635)(2.36,6.635)
\psline[linecolor=black, linewidth=0.04](0.76,9.035)(2.36,9.035)(2.36,9.035)
\psline[linecolor=black, linewidth=0.04](0.76,6.635)(2.36,6.635)(2.36,6.635)
\rput[bl](5.08,9.255){1}
\rput[bl](5.86,9.235){2}
\rput[bl](6.56,8.135){3}
\rput[bl](5.9,6.155){5}
\rput[bl](5.06,6.155){6}
\rput[bl](6.56,7.315){4}
\rput[bl](4.26,6.175){7}
\rput[bl](3.6,7.315){8}
\rput[bl](3.6,8.095){9}
\rput[bl](4.12,9.255){10}
\psdots[linecolor=black, dotsize=0.2](4.36,9.035)
\psdots[linecolor=black, dotsize=0.2](5.16,9.035)
\psdots[linecolor=black, dotsize=0.2](5.96,9.035)
\psdots[linecolor=black, dotsize=0.2](6.36,8.235)
\psdots[linecolor=black, dotsize=0.2](6.36,7.435)
\psdots[linecolor=black, dotsize=0.2](5.96,6.635)
\psdots[linecolor=black, dotsize=0.2](5.16,6.635)
\psdots[linecolor=black, dotsize=0.2](4.36,6.635)
\psdots[linecolor=black, dotsize=0.2](3.96,7.435)
\psdots[linecolor=black, dotsize=0.2](3.96,8.235)
\psline[linecolor=black, linewidth=0.04](4.36,9.035)(3.96,8.235)(3.96,7.435)(4.36,6.635)(4.36,6.635)
\psline[linecolor=black, linewidth=0.04](5.96,9.035)(6.36,8.235)(6.36,7.435)(5.96,6.635)(5.96,6.635)
\psline[linecolor=black, linewidth=0.04](4.36,9.035)(5.96,9.035)(5.96,9.035)
\psline[linecolor=black, linewidth=0.04](4.36,6.635)(5.96,6.635)(5.96,6.635)
\rput[bl](8.68,9.255){1}
\rput[bl](9.46,9.235){2}
\rput[bl](10.16,8.135){3}
\rput[bl](9.5,6.155){5}
\rput[bl](8.66,6.155){6}
\rput[bl](10.16,7.315){4}
\rput[bl](7.86,6.175){7}
\rput[bl](7.2,7.315){8}
\rput[bl](7.2,8.095){9}
\rput[bl](7.72,9.255){10}
\psdots[linecolor=black, dotsize=0.2](7.96,9.035)
\psdots[linecolor=black, dotsize=0.2](8.76,9.035)
\psdots[linecolor=black, dotsize=0.2](9.56,9.035)
\psdots[linecolor=black, dotsize=0.2](9.96,8.235)
\psdots[linecolor=black, dotsize=0.2](9.96,7.435)
\psdots[linecolor=black, dotsize=0.2](9.56,6.635)
\psdots[linecolor=black, dotsize=0.2](8.76,6.635)
\psdots[linecolor=black, dotsize=0.2](7.96,6.635)
\psdots[linecolor=black, dotsize=0.2](7.56,7.435)
\psdots[linecolor=black, dotsize=0.2](7.56,8.235)
\psline[linecolor=black, linewidth=0.04](7.96,9.035)(7.56,8.235)(7.56,7.435)(7.96,6.635)(7.96,6.635)
\psline[linecolor=black, linewidth=0.04](9.56,9.035)(9.96,8.235)(9.96,7.435)(9.56,6.635)(9.56,6.635)
\psline[linecolor=black, linewidth=0.04](7.96,9.035)(9.56,9.035)(9.56,9.035)
\psline[linecolor=black, linewidth=0.04](7.96,6.635)(9.56,6.635)(9.56,6.635)
\rput[bl](12.28,9.255){1}
\rput[bl](13.06,9.235){2}
\rput[bl](13.76,8.135){3}
\rput[bl](13.1,6.155){5}
\rput[bl](12.26,6.155){6}
\rput[bl](13.76,7.315){4}
\rput[bl](11.46,6.175){7}
\rput[bl](10.8,7.315){8}
\rput[bl](10.8,8.095){9}
\rput[bl](11.32,9.255){10}
\psdots[linecolor=black, dotsize=0.2](11.56,9.035)
\psdots[linecolor=black, dotsize=0.2](12.36,9.035)
\psdots[linecolor=black, dotsize=0.2](13.16,9.035)
\psdots[linecolor=black, dotsize=0.2](13.56,8.235)
\psdots[linecolor=black, dotsize=0.2](13.56,7.435)
\psdots[linecolor=black, dotsize=0.2](13.16,6.635)
\psdots[linecolor=black, dotsize=0.2](12.36,6.635)
\psdots[linecolor=black, dotsize=0.2](11.56,6.635)
\psdots[linecolor=black, dotsize=0.2](11.16,7.435)
\psdots[linecolor=black, dotsize=0.2](11.16,8.235)
\psline[linecolor=black, linewidth=0.04](11.56,9.035)(11.16,8.235)(11.16,7.435)(11.56,6.635)(11.56,6.635)
\psline[linecolor=black, linewidth=0.04](13.16,9.035)(13.56,8.235)(13.56,7.435)(13.16,6.635)(13.16,6.635)
\psline[linecolor=black, linewidth=0.04](11.56,9.035)(13.16,9.035)(13.16,9.035)
\psline[linecolor=black, linewidth=0.04](11.56,6.635)(13.16,6.635)(13.16,6.635)
\rput[bl](15.88,9.255){1}
\rput[bl](16.66,9.235){2}
\rput[bl](17.36,8.135){3}
\rput[bl](16.7,6.155){5}
\rput[bl](15.86,6.155){6}
\rput[bl](17.36,7.315){4}
\rput[bl](15.06,6.175){7}
\rput[bl](14.4,7.315){8}
\rput[bl](14.4,8.095){9}
\rput[bl](14.92,9.255){10}
\psdots[linecolor=black, dotsize=0.2](15.16,9.035)
\psdots[linecolor=black, dotsize=0.2](15.96,9.035)
\psdots[linecolor=black, dotsize=0.2](16.76,9.035)
\psdots[linecolor=black, dotsize=0.2](17.16,8.235)
\psdots[linecolor=black, dotsize=0.2](17.16,7.435)
\psdots[linecolor=black, dotsize=0.2](16.76,6.635)
\psdots[linecolor=black, dotsize=0.2](15.96,6.635)
\psdots[linecolor=black, dotsize=0.2](15.16,6.635)
\psdots[linecolor=black, dotsize=0.2](14.76,7.435)
\psdots[linecolor=black, dotsize=0.2](14.76,8.235)
\psline[linecolor=black, linewidth=0.04](15.16,9.035)(14.76,8.235)(14.76,7.435)(15.16,6.635)(15.16,6.635)
\psline[linecolor=black, linewidth=0.04](16.76,9.035)(17.16,8.235)(17.16,7.435)(16.76,6.635)(16.76,6.635)
\psline[linecolor=black, linewidth=0.04](15.16,9.035)(16.76,9.035)(16.76,9.035)
\psline[linecolor=black, linewidth=0.04](15.16,6.635)(16.76,6.635)(16.76,6.635)
\rput[bl](19.48,9.255){1}
\rput[bl](20.26,9.235){2}
\rput[bl](20.96,8.135){3}
\rput[bl](20.3,6.155){5}
\rput[bl](19.46,6.155){6}
\rput[bl](20.96,7.315){4}
\rput[bl](18.66,6.175){7}
\rput[bl](18.0,7.315){8}
\rput[bl](18.0,8.095){9}
\rput[bl](18.52,9.255){10}
\psdots[linecolor=black, dotsize=0.2](18.76,9.035)
\psdots[linecolor=black, dotsize=0.2](19.56,9.035)
\psdots[linecolor=black, dotsize=0.2](20.36,9.035)
\psdots[linecolor=black, dotsize=0.2](20.76,8.235)
\psdots[linecolor=black, dotsize=0.2](20.76,7.435)
\psdots[linecolor=black, dotsize=0.2](20.36,6.635)
\psdots[linecolor=black, dotsize=0.2](19.56,6.635)
\psdots[linecolor=black, dotsize=0.2](18.76,6.635)
\psdots[linecolor=black, dotsize=0.2](18.36,7.435)
\psdots[linecolor=black, dotsize=0.2](18.36,8.235)
\psline[linecolor=black, linewidth=0.04](18.76,9.035)(18.36,8.235)(18.36,7.435)(18.76,6.635)(18.76,6.635)
\psline[linecolor=black, linewidth=0.04](20.36,9.035)(20.76,8.235)(20.76,7.435)(20.36,6.635)(20.36,6.635)
\psline[linecolor=black, linewidth=0.04](18.76,9.035)(20.36,9.035)(20.36,9.035)
\psline[linecolor=black, linewidth=0.04](18.76,6.635)(20.36,6.635)(20.36,6.635)
\rput[bl](1.48,4.055){1}
\rput[bl](2.26,4.035){2}
\rput[bl](2.96,2.935){3}
\rput[bl](2.3,0.955){5}
\rput[bl](1.46,0.955){6}
\rput[bl](2.96,2.115){4}
\rput[bl](0.66,0.975){7}
\rput[bl](0.0,2.115){8}
\rput[bl](0.0,2.895){9}
\rput[bl](0.52,4.055){10}
\psdots[linecolor=black, dotsize=0.2](0.76,3.835)
\psdots[linecolor=black, dotsize=0.2](1.56,3.835)
\psdots[linecolor=black, dotsize=0.2](2.36,3.835)
\psdots[linecolor=black, dotsize=0.2](2.76,3.035)
\psdots[linecolor=black, dotsize=0.2](2.76,2.235)
\psdots[linecolor=black, dotsize=0.2](2.36,1.435)
\psdots[linecolor=black, dotsize=0.2](1.56,1.435)
\psdots[linecolor=black, dotsize=0.2](0.76,1.435)
\psdots[linecolor=black, dotsize=0.2](0.36,2.235)
\psdots[linecolor=black, dotsize=0.2](0.36,3.035)
\psline[linecolor=black, linewidth=0.04](0.76,3.835)(0.36,3.035)(0.36,2.235)(0.76,1.435)(0.76,1.435)
\psline[linecolor=black, linewidth=0.04](2.36,3.835)(2.76,3.035)(2.76,2.235)(2.36,1.435)(2.36,1.435)
\psline[linecolor=black, linewidth=0.04](0.76,3.835)(2.36,3.835)(2.36,3.835)
\psline[linecolor=black, linewidth=0.04](0.76,1.435)(2.36,1.435)(2.36,1.435)
\rput[bl](5.08,4.055){1}
\rput[bl](5.86,4.035){2}
\rput[bl](6.56,2.935){3}
\rput[bl](5.9,0.955){5}
\rput[bl](5.06,0.955){6}
\rput[bl](6.56,2.115){4}
\rput[bl](4.26,0.975){7}
\rput[bl](3.6,2.115){8}
\rput[bl](3.6,2.895){9}
\rput[bl](4.12,4.055){10}
\psdots[linecolor=black, dotsize=0.2](4.36,3.835)
\psdots[linecolor=black, dotsize=0.2](5.16,3.835)
\psdots[linecolor=black, dotsize=0.2](5.96,3.835)
\psdots[linecolor=black, dotsize=0.2](6.36,3.035)
\psdots[linecolor=black, dotsize=0.2](6.36,2.235)
\psdots[linecolor=black, dotsize=0.2](5.96,1.435)
\psdots[linecolor=black, dotsize=0.2](5.16,1.435)
\psdots[linecolor=black, dotsize=0.2](4.36,1.435)
\psdots[linecolor=black, dotsize=0.2](3.96,2.235)
\psdots[linecolor=black, dotsize=0.2](3.96,3.035)
\psline[linecolor=black, linewidth=0.04](4.36,3.835)(3.96,3.035)(3.96,2.235)(4.36,1.435)(4.36,1.435)
\psline[linecolor=black, linewidth=0.04](5.96,3.835)(6.36,3.035)(6.36,2.235)(5.96,1.435)(5.96,1.435)
\psline[linecolor=black, linewidth=0.04](4.36,3.835)(5.96,3.835)(5.96,3.835)
\psline[linecolor=black, linewidth=0.04](4.36,1.435)(5.96,1.435)(5.96,1.435)
\rput[bl](8.68,4.055){1}
\rput[bl](9.46,4.035){2}
\rput[bl](10.16,2.935){3}
\rput[bl](9.5,0.955){5}
\rput[bl](8.66,0.955){6}
\rput[bl](10.16,2.115){4}
\rput[bl](7.86,0.975){7}
\rput[bl](7.2,2.115){8}
\rput[bl](7.2,2.895){9}
\rput[bl](7.72,4.055){10}
\psdots[linecolor=black, dotsize=0.2](7.96,3.835)
\psdots[linecolor=black, dotsize=0.2](8.76,3.835)
\psdots[linecolor=black, dotsize=0.2](9.56,3.835)
\psdots[linecolor=black, dotsize=0.2](9.96,3.035)
\psdots[linecolor=black, dotsize=0.2](9.96,2.235)
\psdots[linecolor=black, dotsize=0.2](9.56,1.435)
\psdots[linecolor=black, dotsize=0.2](8.76,1.435)
\psdots[linecolor=black, dotsize=0.2](7.96,1.435)
\psdots[linecolor=black, dotsize=0.2](7.56,2.235)
\psdots[linecolor=black, dotsize=0.2](7.56,3.035)
\psline[linecolor=black, linewidth=0.04](7.96,3.835)(7.56,3.035)(7.56,2.235)(7.96,1.435)(7.96,1.435)
\psline[linecolor=black, linewidth=0.04](9.56,3.835)(9.96,3.035)(9.96,2.235)(9.56,1.435)(9.56,1.435)
\psline[linecolor=black, linewidth=0.04](7.96,3.835)(9.56,3.835)(9.56,3.835)
\psline[linecolor=black, linewidth=0.04](7.96,1.435)(9.56,1.435)(9.56,1.435)
\rput[bl](12.28,4.055){1}
\rput[bl](13.06,4.035){2}
\rput[bl](13.76,2.935){3}
\rput[bl](13.1,0.955){5}
\rput[bl](12.26,0.955){6}
\rput[bl](13.76,2.115){4}
\rput[bl](11.46,0.975){7}
\rput[bl](10.8,2.115){8}
\rput[bl](10.8,2.895){9}
\rput[bl](11.32,4.055){10}
\psdots[linecolor=black, dotsize=0.2](11.56,3.835)
\psdots[linecolor=black, dotsize=0.2](12.36,3.835)
\psdots[linecolor=black, dotsize=0.2](13.16,3.835)
\psdots[linecolor=black, dotsize=0.2](13.56,3.035)
\psdots[linecolor=black, dotsize=0.2](13.56,2.235)
\psdots[linecolor=black, dotsize=0.2](13.16,1.435)
\psdots[linecolor=black, dotsize=0.2](12.36,1.435)
\psdots[linecolor=black, dotsize=0.2](11.56,1.435)
\psdots[linecolor=black, dotsize=0.2](11.16,2.235)
\psdots[linecolor=black, dotsize=0.2](11.16,3.035)
\rput[bl](15.88,4.055){1}
\rput[bl](16.66,4.035){2}
\rput[bl](17.36,2.935){3}
\rput[bl](16.7,0.955){5}
\rput[bl](15.86,0.955){6}
\rput[bl](17.36,2.115){4}
\rput[bl](15.06,0.975){7}
\rput[bl](14.4,2.115){8}
\rput[bl](14.4,2.895){9}
\rput[bl](14.92,4.055){10}
\psdots[linecolor=black, dotsize=0.2](15.16,3.835)
\psdots[linecolor=black, dotsize=0.2](15.96,3.835)
\psdots[linecolor=black, dotsize=0.2](16.76,3.835)
\psdots[linecolor=black, dotsize=0.2](17.16,3.035)
\psdots[linecolor=black, dotsize=0.2](17.16,2.235)
\psdots[linecolor=black, dotsize=0.2](16.76,1.435)
\psdots[linecolor=black, dotsize=0.2](15.96,1.435)
\psdots[linecolor=black, dotsize=0.2](15.16,1.435)
\psdots[linecolor=black, dotsize=0.2](14.76,2.235)
\psdots[linecolor=black, dotsize=0.2](14.76,3.035)
\psline[linecolor=black, linewidth=0.04](15.16,3.835)(14.76,3.035)(14.76,2.235)(15.16,1.435)(15.16,1.435)
\psline[linecolor=black, linewidth=0.04](16.76,3.835)(17.16,3.035)(17.16,2.235)(16.76,1.435)(16.76,1.435)
\psline[linecolor=black, linewidth=0.04](15.16,3.835)(16.76,3.835)(16.76,3.835)
\psline[linecolor=black, linewidth=0.04](15.16,1.435)(16.76,1.435)(16.76,1.435)
\rput[bl](19.48,4.055){1}
\rput[bl](20.26,4.035){2}
\rput[bl](20.96,2.935){3}
\rput[bl](20.3,0.955){5}
\rput[bl](19.46,0.955){6}
\rput[bl](20.96,2.115){4}
\rput[bl](18.66,0.975){7}
\rput[bl](18.0,2.115){8}
\rput[bl](18.0,2.895){9}
\rput[bl](18.52,4.055){10}
\psdots[linecolor=black, dotsize=0.2](18.76,3.835)
\psdots[linecolor=black, dotsize=0.2](19.56,3.835)
\psdots[linecolor=black, dotsize=0.2](20.36,3.835)
\psdots[linecolor=black, dotsize=0.2](20.76,3.035)
\psdots[linecolor=black, dotsize=0.2](20.76,2.235)
\psdots[linecolor=black, dotsize=0.2](20.36,1.435)
\psdots[linecolor=black, dotsize=0.2](19.56,1.435)
\psdots[linecolor=black, dotsize=0.2](18.76,1.435)
\psdots[linecolor=black, dotsize=0.2](18.36,2.235)
\psdots[linecolor=black, dotsize=0.2](18.36,3.035)
\psline[linecolor=black, linewidth=0.04](18.76,3.835)(18.36,3.035)(18.36,2.235)(18.76,1.435)(18.76,1.435)
\psline[linecolor=black, linewidth=0.04](20.36,3.835)(20.76,3.035)(20.76,2.235)(20.36,1.435)(20.36,1.435)
\psline[linecolor=black, linewidth=0.04](18.76,3.835)(20.36,3.835)(20.36,3.835)
\psline[linecolor=black, linewidth=0.04](18.76,1.435)(20.36,1.435)(20.36,1.435)
\rput[bl](1.48,-1.145){1}
\rput[bl](2.26,-1.165){2}
\rput[bl](2.96,-2.265){3}
\rput[bl](2.3,-4.245){5}
\rput[bl](1.46,-4.245){6}
\rput[bl](2.96,-3.085){4}
\rput[bl](0.66,-4.225){7}
\rput[bl](0.0,-3.085){8}
\rput[bl](0.0,-2.305){9}
\rput[bl](0.52,-1.145){10}
\psdots[linecolor=black, dotsize=0.2](0.76,-1.365)
\psdots[linecolor=black, dotsize=0.2](1.56,-1.365)
\psdots[linecolor=black, dotsize=0.2](2.36,-1.365)
\psdots[linecolor=black, dotsize=0.2](2.76,-2.165)
\psdots[linecolor=black, dotsize=0.2](2.76,-2.965)
\psdots[linecolor=black, dotsize=0.2](2.36,-3.765)
\psdots[linecolor=black, dotsize=0.2](1.56,-3.765)
\psdots[linecolor=black, dotsize=0.2](0.76,-3.765)
\psdots[linecolor=black, dotsize=0.2](0.36,-2.965)
\psdots[linecolor=black, dotsize=0.2](0.36,-2.165)
\psline[linecolor=black, linewidth=0.04](0.76,-1.365)(0.36,-2.165)(0.36,-2.965)(0.76,-3.765)(0.76,-3.765)
\psline[linecolor=black, linewidth=0.04](2.36,-1.365)(2.76,-2.165)(2.76,-2.965)(2.36,-3.765)(2.36,-3.765)
\psline[linecolor=black, linewidth=0.04](0.76,-1.365)(2.36,-1.365)(2.36,-1.365)
\psline[linecolor=black, linewidth=0.04](0.76,-3.765)(2.36,-3.765)(2.36,-3.765)
\rput[bl](5.08,-1.145){1}
\rput[bl](5.86,-1.165){2}
\rput[bl](6.56,-2.265){3}
\rput[bl](5.9,-4.245){5}
\rput[bl](5.06,-4.245){6}
\rput[bl](6.56,-3.085){4}
\rput[bl](4.26,-4.225){7}
\rput[bl](3.6,-3.085){8}
\rput[bl](3.6,-2.305){9}
\rput[bl](4.12,-1.145){10}
\psdots[linecolor=black, dotsize=0.2](4.36,-1.365)
\psdots[linecolor=black, dotsize=0.2](5.16,-1.365)
\psdots[linecolor=black, dotsize=0.2](5.96,-1.365)
\psdots[linecolor=black, dotsize=0.2](6.36,-2.165)
\psdots[linecolor=black, dotsize=0.2](6.36,-2.965)
\psdots[linecolor=black, dotsize=0.2](5.96,-3.765)
\psdots[linecolor=black, dotsize=0.2](5.16,-3.765)
\psdots[linecolor=black, dotsize=0.2](4.36,-3.765)
\psdots[linecolor=black, dotsize=0.2](3.96,-2.965)
\psdots[linecolor=black, dotsize=0.2](3.96,-2.165)
\psline[linecolor=black, linewidth=0.04](4.36,-1.365)(3.96,-2.165)(3.96,-2.965)(4.36,-3.765)(4.36,-3.765)
\psline[linecolor=black, linewidth=0.04](5.96,-1.365)(6.36,-2.165)(6.36,-2.965)(5.96,-3.765)(5.96,-3.765)
\psline[linecolor=black, linewidth=0.04](4.36,-1.365)(5.96,-1.365)(5.96,-1.365)
\psline[linecolor=black, linewidth=0.04](4.36,-3.765)(5.96,-3.765)(5.96,-3.765)
\rput[bl](8.68,-1.145){1}
\rput[bl](9.46,-1.165){2}
\rput[bl](10.16,-2.265){3}
\rput[bl](9.5,-4.245){5}
\rput[bl](8.66,-4.245){6}
\rput[bl](10.16,-3.085){4}
\rput[bl](7.86,-4.225){7}
\rput[bl](7.2,-3.085){8}
\rput[bl](7.2,-2.305){9}
\rput[bl](7.72,-1.145){10}
\psdots[linecolor=black, dotsize=0.2](7.96,-1.365)
\psdots[linecolor=black, dotsize=0.2](8.76,-1.365)
\psdots[linecolor=black, dotsize=0.2](9.56,-1.365)
\psdots[linecolor=black, dotsize=0.2](9.96,-2.165)
\psdots[linecolor=black, dotsize=0.2](9.96,-2.965)
\psdots[linecolor=black, dotsize=0.2](9.56,-3.765)
\psdots[linecolor=black, dotsize=0.2](8.76,-3.765)
\psdots[linecolor=black, dotsize=0.2](7.96,-3.765)
\psdots[linecolor=black, dotsize=0.2](7.56,-2.965)
\psdots[linecolor=black, dotsize=0.2](7.56,-2.165)
\psline[linecolor=black, linewidth=0.04](7.96,-1.365)(7.56,-2.165)(7.56,-2.965)(7.96,-3.765)(7.96,-3.765)
\psline[linecolor=black, linewidth=0.04](9.56,-1.365)(9.96,-2.165)(9.96,-2.965)(9.56,-3.765)(9.56,-3.765)
\psline[linecolor=black, linewidth=0.04](7.96,-1.365)(9.56,-1.365)(9.56,-1.365)
\psline[linecolor=black, linewidth=0.04](7.96,-3.765)(9.56,-3.765)(9.56,-3.765)
\rput[bl](12.28,-1.145){1}
\rput[bl](13.06,-1.165){2}
\rput[bl](13.76,-2.265){3}
\rput[bl](13.1,-4.245){5}
\rput[bl](12.26,-4.245){6}
\rput[bl](13.76,-3.085){4}
\rput[bl](11.46,-4.225){7}
\rput[bl](10.8,-3.085){8}
\rput[bl](10.8,-2.305){9}
\rput[bl](11.32,-1.145){10}
\psdots[linecolor=black, dotsize=0.2](11.56,-1.365)
\psdots[linecolor=black, dotsize=0.2](12.36,-1.365)
\psdots[linecolor=black, dotsize=0.2](13.16,-1.365)
\psdots[linecolor=black, dotsize=0.2](13.56,-2.165)
\psdots[linecolor=black, dotsize=0.2](13.56,-2.965)
\psdots[linecolor=black, dotsize=0.2](13.16,-3.765)
\psdots[linecolor=black, dotsize=0.2](12.36,-3.765)
\psdots[linecolor=black, dotsize=0.2](11.56,-3.765)
\psdots[linecolor=black, dotsize=0.2](11.16,-2.965)
\psdots[linecolor=black, dotsize=0.2](11.16,-2.165)
\psline[linecolor=black, linewidth=0.04](11.56,-1.365)(11.16,-2.165)(11.16,-2.965)(11.56,-3.765)(11.56,-3.765)
\psline[linecolor=black, linewidth=0.04](13.16,-1.365)(13.56,-2.165)(13.56,-2.965)(13.16,-3.765)(13.16,-3.765)
\psline[linecolor=black, linewidth=0.04](11.56,-1.365)(13.16,-1.365)(13.16,-1.365)
\psline[linecolor=black, linewidth=0.04](11.56,-3.765)(13.16,-3.765)(13.16,-3.765)
\rput[bl](15.88,-1.145){1}
\rput[bl](16.66,-1.165){2}
\rput[bl](17.36,-2.265){3}
\rput[bl](16.7,-4.245){5}
\rput[bl](15.86,-4.245){6}
\rput[bl](17.36,-3.085){4}
\rput[bl](15.06,-4.225){7}
\rput[bl](14.4,-3.085){8}
\rput[bl](14.4,-2.305){9}
\rput[bl](14.92,-1.145){10}
\psdots[linecolor=black, dotsize=0.2](15.16,-1.365)
\psdots[linecolor=black, dotsize=0.2](15.96,-1.365)
\psdots[linecolor=black, dotsize=0.2](16.76,-1.365)
\psdots[linecolor=black, dotsize=0.2](17.16,-2.165)
\psdots[linecolor=black, dotsize=0.2](17.16,-2.965)
\psdots[linecolor=black, dotsize=0.2](16.76,-3.765)
\psdots[linecolor=black, dotsize=0.2](15.96,-3.765)
\psdots[linecolor=black, dotsize=0.2](15.16,-3.765)
\psdots[linecolor=black, dotsize=0.2](14.76,-2.965)
\psdots[linecolor=black, dotsize=0.2](14.76,-2.165)
\rput[bl](19.48,-1.145){1}
\rput[bl](20.26,-1.165){2}
\rput[bl](20.96,-2.265){3}
\rput[bl](20.3,-4.245){5}
\rput[bl](19.46,-4.245){6}
\rput[bl](20.96,-3.085){4}
\rput[bl](18.66,-4.225){7}
\rput[bl](18.0,-3.085){8}
\rput[bl](18.0,-2.305){9}
\rput[bl](18.52,-1.145){10}
\psdots[linecolor=black, dotsize=0.2](18.76,-1.365)
\psdots[linecolor=black, dotsize=0.2](19.56,-1.365)
\psdots[linecolor=black, dotsize=0.2](20.36,-1.365)
\psdots[linecolor=black, dotsize=0.2](20.76,-2.165)
\psdots[linecolor=black, dotsize=0.2](20.76,-2.965)
\psdots[linecolor=black, dotsize=0.2](20.36,-3.765)
\psdots[linecolor=black, dotsize=0.2](19.56,-3.765)
\psdots[linecolor=black, dotsize=0.2](18.76,-3.765)
\psdots[linecolor=black, dotsize=0.2](18.36,-2.965)
\psdots[linecolor=black, dotsize=0.2](18.36,-2.165)
\psline[linecolor=black, linewidth=0.04](18.76,-1.365)(18.36,-2.165)(18.36,-2.965)(18.76,-3.765)(18.76,-3.765)
\psline[linecolor=black, linewidth=0.04](20.36,-1.365)(20.76,-2.165)(20.76,-2.965)(20.36,-3.765)(20.36,-3.765)
\psline[linecolor=black, linewidth=0.04](18.76,-3.765)(20.36,-3.765)(20.36,-3.765)
\rput[bl](4.68,-6.345){1}
\rput[bl](5.46,-6.365){2}
\rput[bl](6.16,-7.465){3}
\rput[bl](5.5,-9.445){5}
\rput[bl](4.66,-9.445){6}
\rput[bl](6.16,-8.285){4}
\rput[bl](3.86,-9.425){7}
\rput[bl](3.2,-8.285){8}
\rput[bl](3.2,-7.505){9}
\rput[bl](3.72,-6.345){10}
\psdots[linecolor=black, dotsize=0.2](3.96,-6.565)
\psdots[linecolor=black, dotsize=0.2](4.76,-6.565)
\psdots[linecolor=black, dotsize=0.2](5.56,-6.565)
\psdots[linecolor=black, dotsize=0.2](5.96,-7.365)
\psdots[linecolor=black, dotsize=0.2](5.96,-8.165)
\psdots[linecolor=black, dotsize=0.2](5.56,-8.965)
\psdots[linecolor=black, dotsize=0.2](4.76,-8.965)
\psdots[linecolor=black, dotsize=0.2](3.96,-8.965)
\psdots[linecolor=black, dotsize=0.2](3.56,-8.165)
\psdots[linecolor=black, dotsize=0.2](3.56,-7.365)
\psline[linecolor=black, linewidth=0.04](3.96,-6.565)(3.56,-7.365)(3.56,-8.165)(3.96,-8.965)(3.96,-8.965)
\psline[linecolor=black, linewidth=0.04](5.56,-6.565)(5.96,-7.365)(5.96,-8.165)(5.56,-8.965)(5.56,-8.965)
\psline[linecolor=black, linewidth=0.04](3.96,-6.565)(5.56,-6.565)(5.56,-6.565)
\psline[linecolor=black, linewidth=0.04](3.96,-8.965)(5.56,-8.965)(5.56,-8.965)
\rput[bl](9.88,-6.345){1}
\rput[bl](10.66,-7.165){2}
\rput[bl](10.62,-8.625){3}
\rput[bl](8.8,-8.265){5}
\rput[bl](8.8,-7.505){6}
\psrotate(9.82, -9.425){-0.37204528}{\rput[bl](9.82,-9.425){4}}
\rput[bl](11.04,-7.165){7}
\rput[bl](11.84,-7.145){8}
\rput[bl](11.86,-8.625){9}
\rput[bl](11.02,-8.585){10}
\rput[bl](1.4,0.255){\large{$G_7$}}
\rput[bl](5.0,0.255){\large{$G_8$}}
\rput[bl](8.6,0.255){\large{$G_9$}}
\rput[bl](12.2,0.255){\large{$G_{10}$}}
\rput[bl](15.8,0.255){\large{$G_{11}$}}
\rput[bl](19.4,0.255){\large{$G_{12}$}}
\rput[bl](1.4,-4.945){\large{$G_{13}$}}
\rput[bl](5.0,-4.945){\large{$G_{14}$}}
\rput[bl](8.6,-4.945){\large{$G_{15}$}}
\rput[bl](12.2,-4.945){\large{$G_{16}$}}
\rput[bl](15.8,-4.945){\large{$G_{17}$}}
\rput[bl](19.4,-4.945){\large{$G_{18}$}}
\rput[bl](4.6,-10.145){\large{$G_{19}$}}
\rput[bl](10.2,-10.145){\large{$G_{20}$}}
\rput[bl](15.8,-10.145){\large{$G_{21}$}}
\psline[linecolor=black, linewidth=0.04](1.56,9.035)(1.56,6.635)(1.56,6.635)
\psline[linecolor=black, linewidth=0.04](2.36,9.035)(2.76,7.435)(2.76,7.435)
\psline[linecolor=black, linewidth=0.04](2.76,8.235)(2.36,6.635)(2.36,6.635)
\psline[linecolor=black, linewidth=0.04](0.76,9.035)(0.36,7.435)(0.36,7.435)
\psline[linecolor=black, linewidth=0.04](0.36,8.235)(0.76,6.635)(0.76,6.635)
\psline[linecolor=black, linewidth=0.04](5.16,9.035)(5.16,6.635)
\psline[linecolor=black, linewidth=0.04](5.96,9.035)(3.96,8.235)(3.96,8.235)
\psline[linecolor=black, linewidth=0.04](4.36,9.035)(3.96,7.435)(3.96,7.435)
\psline[linecolor=black, linewidth=0.04](6.36,8.235)(5.96,6.635)(5.96,6.635)
\psline[linecolor=black, linewidth=0.04](4.36,6.635)(6.36,7.435)(6.36,7.435)
\psbezier[linecolor=black, linewidth=0.04](7.96,9.035)(7.96,8.235)(9.56,8.235)(9.56,9.035)
\psline[linecolor=black, linewidth=0.04](7.56,8.235)(9.96,8.235)(9.96,8.235)
\psline[linecolor=black, linewidth=0.04](8.76,9.035)(9.96,7.435)(9.96,7.435)
\psline[linecolor=black, linewidth=0.04](8.76,6.635)(7.56,7.435)(7.56,7.435)
\psbezier[linecolor=black, linewidth=0.04](11.56,9.035)(11.56,8.235)(13.16,8.235)(13.16,9.035)
\psline[linecolor=black, linewidth=0.04](12.36,9.035)(13.56,8.235)(13.56,8.235)
\psline[linecolor=black, linewidth=0.04](13.56,7.435)(11.56,6.635)(11.56,6.635)
\psline[linecolor=black, linewidth=0.04](13.16,6.635)(11.16,7.435)(11.16,7.435)
\psline[linecolor=black, linewidth=0.04](12.36,6.635)(11.16,8.235)(11.16,8.235)
\psline[linecolor=black, linewidth=0.04](15.96,9.035)(17.16,8.235)(17.16,8.235)
\psbezier[linecolor=black, linewidth=0.04](15.16,9.035)(15.16,8.235)(16.76,8.235)(16.76,9.035)
\psline[linecolor=black, linewidth=0.04](17.16,7.435)(15.96,6.635)(15.96,6.635)
\psline[linecolor=black, linewidth=0.04](16.76,6.635)(14.76,7.435)(14.76,7.435)
\psline[linecolor=black, linewidth=0.04](15.16,6.635)(14.76,8.235)(14.76,8.235)
\psline[linecolor=black, linewidth=0.04](19.56,9.035)(19.56,6.635)(19.56,6.635)
\psline[linecolor=black, linewidth=0.04](20.36,9.035)(18.76,6.635)(18.76,6.635)
\psline[linecolor=black, linewidth=0.04](20.76,8.235)(18.36,7.435)(18.36,7.435)
\psline[linecolor=black, linewidth=0.04](20.76,7.435)(18.36,8.235)(18.36,8.235)
\psline[linecolor=black, linewidth=0.04](18.76,9.035)(20.36,6.635)(20.36,6.635)
\psline[linecolor=black, linewidth=0.04](1.56,3.835)(1.56,1.435)(1.56,1.435)
\psline[linecolor=black, linewidth=0.04](0.76,1.435)(2.36,3.835)(2.36,3.835)
\psline[linecolor=black, linewidth=0.04](2.36,1.435)(0.76,3.835)(0.76,3.835)
\psline[linecolor=black, linewidth=0.04](0.36,3.035)(2.76,3.035)(2.76,3.035)
\psline[linecolor=black, linewidth=0.04](0.36,2.235)(2.76,2.235)(2.76,2.235)
\psline[linecolor=black, linewidth=0.04](5.16,3.835)(5.16,1.435)(5.16,1.435)
\psline[linecolor=black, linewidth=0.04](5.96,3.835)(3.96,2.235)(3.96,2.235)
\psline[linecolor=black, linewidth=0.04](6.36,3.035)(4.36,1.435)(4.36,1.435)
\psline[linecolor=black, linewidth=0.04](6.36,2.235)(4.36,3.835)(4.36,3.835)
\psline[linecolor=black, linewidth=0.04](5.96,1.435)(3.96,3.035)(3.96,3.035)
\psline[linecolor=black, linewidth=0.04](8.76,3.835)(8.76,1.435)(8.76,1.435)
\psbezier[linecolor=black, linewidth=0.04](7.96,3.835)(7.96,3.035)(9.56,3.035)(9.56,3.835)
\psline[linecolor=black, linewidth=0.04](7.96,1.435)(9.96,3.035)(9.96,3.035)
\psline[linecolor=black, linewidth=0.04](9.96,2.235)(7.56,2.235)(7.56,2.235)
\psline[linecolor=black, linewidth=0.04](9.56,1.435)(7.56,3.035)(7.56,3.035)
\psline[linecolor=black, linewidth=0.04](12.36,3.835)(13.16,3.835)(13.16,3.835)
\psline[linecolor=black, linewidth=0.04](12.36,3.835)(12.36,1.435)(12.36,1.435)
\psline[linecolor=black, linewidth=0.04](12.36,3.835)(13.16,1.435)(13.16,1.435)
\psline[linecolor=black, linewidth=0.04](13.16,3.835)(13.56,3.035)(13.56,3.035)
\psline[linecolor=black, linewidth=0.04](13.16,3.835)(11.56,1.435)(11.56,1.435)
\psline[linecolor=black, linewidth=0.04](13.56,3.035)(13.56,2.235)(13.56,2.235)
\psline[linecolor=black, linewidth=0.04](13.56,3.035)(11.16,2.235)(11.16,2.235)
\psline[linecolor=black, linewidth=0.04](13.56,2.235)(13.16,1.435)(13.16,1.435)
\psline[linecolor=black, linewidth=0.04](13.56,2.235)(11.16,3.035)(11.16,3.035)
\psline[linecolor=black, linewidth=0.04](13.16,1.435)(11.56,3.835)(11.56,3.835)
\psline[linecolor=black, linewidth=0.04](12.36,1.435)(11.56,1.435)(11.56,1.435)
\psline[linecolor=black, linewidth=0.04](12.36,1.435)(11.56,3.835)(11.56,3.835)
\psline[linecolor=black, linewidth=0.04](11.56,1.435)(11.16,2.235)(11.16,2.235)
\psline[linecolor=black, linewidth=0.04](11.16,2.235)(11.16,3.035)(11.16,3.035)
\psline[linecolor=black, linewidth=0.04](11.16,3.035)(11.56,3.835)(11.56,3.835)
\psline[linecolor=black, linewidth=0.04](15.96,3.835)(15.96,1.435)(15.96,1.435)
\psbezier[linecolor=black, linewidth=0.04](15.16,3.835)(15.16,3.035)(16.76,3.035)(16.76,3.835)
\psline[linecolor=black, linewidth=0.04](17.16,3.035)(14.76,3.035)(14.76,3.035)
\psline[linecolor=black, linewidth=0.04](17.16,2.235)(14.76,2.235)(14.76,2.235)
\psbezier[linecolor=black, linewidth=0.04](7.96,6.635)(7.96,7.435)(9.56,7.435)(9.56,6.635)
\psbezier[linecolor=black, linewidth=0.04](15.16,1.435)(15.16,2.235)(16.76,2.235)(16.76,1.435)
\psline[linecolor=black, linewidth=0.04](19.56,3.835)(19.56,1.435)(19.56,1.435)
\psline[linecolor=black, linewidth=0.04](20.36,3.835)(20.76,2.235)(20.76,2.235)
\psline[linecolor=black, linewidth=0.04](20.76,3.035)(18.76,1.435)(18.76,1.435)
\psline[linecolor=black, linewidth=0.04](20.36,1.435)(18.36,3.035)(18.36,3.035)
\psline[linecolor=black, linewidth=0.04](18.76,3.835)(18.36,2.235)(18.36,2.235)
\psline[linecolor=black, linewidth=0.04](1.56,-1.365)(1.56,-3.765)(1.56,-3.765)
\psline[linecolor=black, linewidth=0.04](2.36,-3.765)(0.36,-2.965)(0.36,-2.965)
\psline[linecolor=black, linewidth=0.04](0.76,-3.765)(2.76,-2.965)(2.76,-2.965)
\psline[linecolor=black, linewidth=0.04](0.36,-2.165)(2.76,-2.165)(2.36,-2.165)
\psbezier[linecolor=black, linewidth=0.04](0.76,-1.365)(0.76,-2.165)(2.36,-2.165)(2.36,-1.365)
\psline[linecolor=black, linewidth=0.04](5.16,-1.365)(5.16,-3.765)(5.16,-3.765)
\psline[linecolor=black, linewidth=0.04](6.36,-2.165)(5.96,-3.765)(5.96,-3.765)
\psline[linecolor=black, linewidth=0.04](3.96,-2.165)(4.36,-3.765)(4.36,-3.765)
\psline[linecolor=black, linewidth=0.04](3.96,-2.965)(6.36,-2.965)(6.36,-2.965)
\psbezier[linecolor=black, linewidth=0.04](4.36,-1.365)(4.36,-2.165)(5.96,-2.165)(5.96,-1.365)
\psline[linecolor=black, linewidth=0.04](8.76,-1.365)(8.76,-3.765)(8.76,-3.765)
\psline[linecolor=black, linewidth=0.04](9.56,-1.365)(9.96,-2.965)(9.96,-2.965)
\psline[linecolor=black, linewidth=0.04](9.96,-2.165)(7.56,-2.165)(7.56,-2.165)
\psline[linecolor=black, linewidth=0.04](7.96,-1.365)(7.96,-3.765)(7.96,-3.765)
\psline[linecolor=black, linewidth=0.04](7.56,-2.965)(9.56,-3.765)(9.56,-3.765)
\psline[linecolor=black, linewidth=0.04](12.36,-1.365)(12.36,-3.765)(12.36,-3.765)
\psline[linecolor=black, linewidth=0.04](13.16,-1.365)(11.16,-2.165)(11.16,-2.165)
\psline[linecolor=black, linewidth=0.04](11.56,-1.365)(13.56,-2.165)(13.56,-2.165)
\psline[linecolor=black, linewidth=0.04](13.56,-2.965)(11.56,-3.765)(11.56,-3.765)
\psline[linecolor=black, linewidth=0.04](13.16,-3.765)(11.16,-2.965)(11.16,-2.965)
\psline[linecolor=black, linewidth=0.04](15.96,-1.365)(16.76,-1.365)(16.76,-1.365)
\psline[linecolor=black, linewidth=0.04](15.96,-1.365)(16.76,-3.765)(16.76,-3.765)
\psline[linecolor=black, linewidth=0.04](15.96,-1.365)(15.96,-3.765)(15.96,-3.765)
\psline[linecolor=black, linewidth=0.04](16.76,-1.365)(17.16,-2.165)(17.16,-2.165)
\psline[linecolor=black, linewidth=0.04](16.76,-1.365)(15.16,-3.765)(15.16,-3.765)
\psline[linecolor=black, linewidth=0.04](17.16,-2.165)(17.16,-2.965)(17.16,-2.965)
\psline[linecolor=black, linewidth=0.04](17.16,-2.165)(14.76,-2.965)(14.76,-2.965)
\psline[linecolor=black, linewidth=0.04](17.16,-2.965)(16.76,-3.765)(16.76,-3.765)
\psline[linecolor=black, linewidth=0.04](17.16,-2.965)(14.76,-2.165)(14.76,-2.165)
\psline[linecolor=black, linewidth=0.04](16.76,-3.765)(15.16,-1.365)(15.16,-1.365)
\psline[linecolor=black, linewidth=0.04](15.96,-3.765)(14.76,-2.965)(14.76,-2.965)
\psline[linecolor=black, linewidth=0.04](15.96,-3.765)(14.76,-2.165)(14.76,-2.165)
\psline[linecolor=black, linewidth=0.04](15.16,-3.765)(14.76,-2.165)(14.76,-2.165)
\psline[linecolor=black, linewidth=0.04](15.16,-3.765)(15.16,-1.365)(15.16,-1.365)
\psline[linecolor=black, linewidth=0.04](14.76,-2.965)(15.16,-1.365)(15.16,-1.365)
\psline[linecolor=black, linewidth=0.04](19.56,-1.365)(20.36,-1.365)(20.36,-1.365)
\psline[linecolor=black, linewidth=0.04](19.56,-1.365)(20.76,-2.165)(20.76,-2.165)
\psline[linecolor=black, linewidth=0.04](19.56,-1.365)(20.36,-3.765)(20.36,-3.765)
\psline[linecolor=black, linewidth=0.04](20.36,-1.365)(20.76,-2.965)(20.76,-2.965)
\psline[linecolor=black, linewidth=0.04](19.56,-3.765)(18.76,-1.365)(18.76,-1.365)
\psline[linecolor=black, linewidth=0.04](18.76,-1.365)(18.36,-2.965)(18.36,-2.965)
\psline[linecolor=black, linewidth=0.04](18.36,-2.165)(18.76,-3.765)(18.76,-3.765)
\psline[linecolor=black, linewidth=0.04](4.76,-6.565)(4.76,-8.965)(4.76,-8.965)
\psline[linecolor=black, linewidth=0.04](5.96,-7.365)(3.56,-8.165)(3.56,-8.165)
\psline[linecolor=black, linewidth=0.04](5.96,-8.165)(3.56,-7.365)(3.56,-7.365)
\psbezier[linecolor=black, linewidth=0.04](3.96,-6.565)(3.96,-7.365)(5.56,-7.365)(5.56,-6.565)
\psbezier[linecolor=black, linewidth=0.04](3.96,-8.965)(3.96,-8.165)(5.56,-8.165)(5.56,-8.965)
\psline[linecolor=black, linewidth=0.04](9.96,-8.965)(9.16,-8.165)(9.16,-7.365)(9.96,-6.565)(10.76,-7.365)(10.76,-8.165)(9.96,-8.965)(9.96,-8.965)
\psline[linecolor=black, linewidth=0.04](11.16,-8.165)(11.16,-7.365)(11.96,-7.365)(11.96,-8.165)(11.16,-8.165)(11.16,-8.165)
\psline[linecolor=black, linewidth=0.04](11.16,-7.365)(11.96,-8.165)(11.96,-8.165)
\psline[linecolor=black, linewidth=0.04](11.96,-7.365)(11.16,-8.165)(11.16,-8.165)
\psdots[linecolor=black, dotsize=0.2](9.96,-6.565)
\psdots[linecolor=black, dotsize=0.2](10.76,-7.365)
\psdots[linecolor=black, dotsize=0.2](10.76,-8.165)
\psdots[linecolor=black, dotsize=0.2](9.96,-8.965)
\psdots[linecolor=black, dotsize=0.2](9.16,-8.165)
\psdots[linecolor=black, dotsize=0.2](9.16,-7.365)
\psdots[linecolor=black, dotsize=0.2](11.16,-7.365)
\psdots[linecolor=black, dotsize=0.2](11.96,-7.365)
\psdots[linecolor=black, dotsize=0.2](11.96,-8.165)
\psdots[linecolor=black, dotsize=0.2](11.16,-8.165)
\rput[bl](15.48,-6.345){1}
\rput[bl](16.26,-7.165){2}
\rput[bl](16.22,-8.625){3}
\rput[bl](14.4,-8.265){5}
\rput[bl](14.4,-7.505){6}
\psrotate(15.42, -9.425){-0.37204528}{\rput[bl](15.42,-9.425){4}}
\rput[bl](16.64,-7.165){7}
\rput[bl](17.44,-7.145){8}
\rput[bl](17.46,-8.625){9}
\rput[bl](16.62,-8.585){10}
\psline[linecolor=black, linewidth=0.04](15.56,-8.965)(14.76,-8.165)(14.76,-7.365)(15.56,-6.565)(16.36,-7.365)(16.36,-8.165)(15.56,-8.965)(15.56,-8.965)
\psline[linecolor=black, linewidth=0.04](16.76,-8.165)(16.76,-7.365)(17.56,-7.365)(17.56,-8.165)(16.76,-8.165)(16.76,-8.165)
\psline[linecolor=black, linewidth=0.04](16.76,-7.365)(17.56,-8.165)(17.56,-8.165)
\psline[linecolor=black, linewidth=0.04](17.56,-7.365)(16.76,-8.165)(16.76,-8.165)
\psdots[linecolor=black, dotsize=0.2](15.56,-6.565)
\psdots[linecolor=black, dotsize=0.2](16.36,-7.365)
\psdots[linecolor=black, dotsize=0.2](16.36,-8.165)
\psdots[linecolor=black, dotsize=0.2](15.56,-8.965)
\psdots[linecolor=black, dotsize=0.2](14.76,-8.165)
\psdots[linecolor=black, dotsize=0.2](14.76,-7.365)
\psdots[linecolor=black, dotsize=0.2](16.76,-7.365)
\psdots[linecolor=black, dotsize=0.2](17.56,-7.365)
\psdots[linecolor=black, dotsize=0.2](17.56,-8.165)
\psdots[linecolor=black, dotsize=0.2](16.76,-8.165)
\psline[linecolor=black, linewidth=0.04](9.96,-6.565)(9.96,-8.965)(9.96,-8.965)
\psline[linecolor=black, linewidth=0.04](9.16,-7.365)(10.76,-7.365)(10.76,-7.365)
\psline[linecolor=black, linewidth=0.04](9.16,-8.165)(10.76,-8.165)(10.76,-8.165)
\psline[linecolor=black, linewidth=0.04](15.56,-6.565)(15.56,-8.965)(15.56,-8.965)
\psline[linecolor=black, linewidth=0.04](16.36,-8.165)(14.76,-7.365)(14.76,-7.365)
\psline[linecolor=black, linewidth=0.04](14.76,-8.165)(16.36,-7.365)(16.36,-7.365)
\end{pspicture}
}
\end{center}
\caption{Cubic graphs of order $10$.}\label{fig:cubic}
	\end{figure}
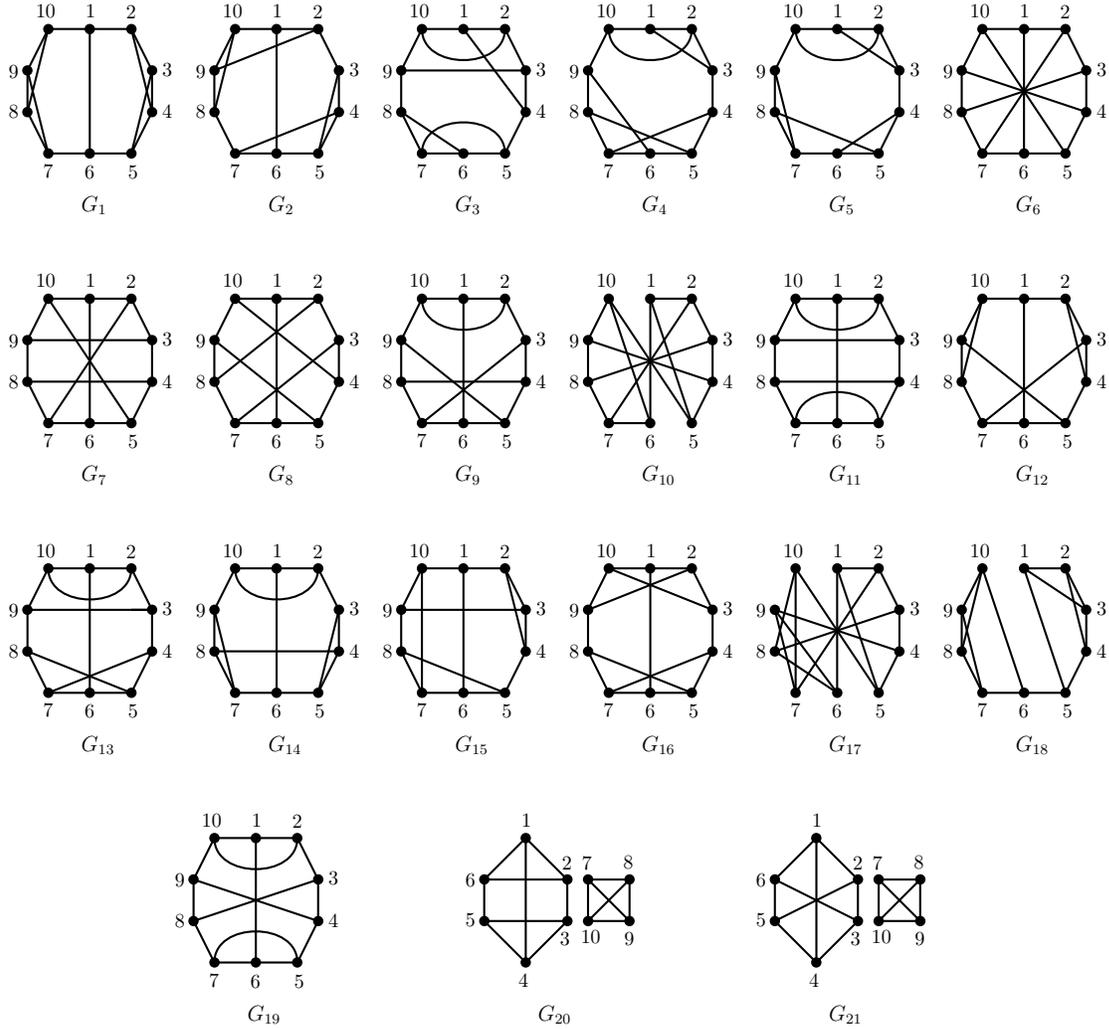

In the following, we consider cubic graphs of order $10$, as we see in Figure \ref{fig:cubic}. Note that $G_{17}=P$.

	\begin{theorem}\label{thm:cubic10}
If $G$ is a cubic graph of order $10$ which is not the Petersen graph, then $\dst(G)=3$.	
	\end{theorem}

	\begin{proof}
Consider Figure \ref{fig:cubic}. Suppose that $D$ is a strong dominating set of a cubic graph of order $10$. Since each vertex in $D$ strong dominate at most $3$ other vertices,  we need to have $|D|\geq 3$. Now, consider the following sets:

\begin{align*}
P_{1} &= \Bigl\{ \{1,3,9 \},\{2,6,8 \},\{4,5,7,10 \} \Bigl\}, &
P_{2} &= \Bigl\{ \{1,3,8 \},\{2,5,7,10 \},\{4,6,9 \} \Bigl\},\\
P_{3} &= \Bigl\{ \{1,3,6 \},\{2,5,9 \},\{4,7,8,10 \} \Bigl\}, &
P_{4} &= \Bigl\{ \{1,6,7 \},\{2,4,9 \},\{3,5,8,10 \} \Bigl\},\\
P_{5} &= \Bigl\{ \{1,4,9 \},\{2,6,7 \},\{3,5,8,10 \} \Bigl\}, &
P_{6} &= \Bigl\{ \{1,4,7\},\{2,5,8 \},\{3,6,9,10 \} \Bigl\},\\
P_{7} &= \Bigl\{ \{1,3,6,9 \},\{2,5,8 \},\{4,7,10 \} \Bigl\}, &
P_{8} &= \Bigl\{ \{1,4,8 \},\{2,5,7,10 \},\{3,6,9 \} \Bigl\},\\
P_{9} &= \Bigl\{ \{1,4,8,10 \},\{2,5,7 \},\{3,6,9 \} \Bigl\}, &
P_{10} &= \Bigl\{ \{1,8,9 \},\{2,5,7,10 \},\{3,4,6 \} \Bigl\},\\
P_{11} &= \Bigl\{ \{1,4,8 \},\{2,5,7,10 \},\{3,6,9 \} \Bigl\}, &
P_{12} &= \Bigl\{ \{1,3,9 \},\{2,5,7,10 \},\{4,6,8 \} \Bigl\},\\
P_{13} &= \Bigl\{ \{1,4,8 \},\{2,5,7,10 \},\{3,6,9 \} \Bigl\}, &
P_{14} &= \Bigl\{ \{1,4,8,10 \},\{2,5,7 \},\{3,6,9 \} \Bigl\},\\
P_{15} &= \Bigl\{ \{1,4,8 \},\{2,5,7,10 \},\{3,6,9 \} \Bigl\}, &
P_{16} &= \Bigl\{ \{1,4,8 \},\{2,5,7,10 \},\{3,6,9 \} \Bigl\},\\
P_{18} &= \Bigl\{ \{1,4,7,10 \},\{2,5,8 \},\{3,6,9 \} \Bigl\}, &
P_{19} &= \Bigl\{ \{1,4,8 \},\{2,5,7,10 \},\{3,6,9 \} \Bigl\},\\
P_{20} &= \Bigl\{ \{1,3,7 \},\{2,4,8 \},\{5,6,9,10 \} \Bigl\}, & 
P_{21} &= \Bigl\{ \{1,4,7 \},\{2,5,8 \},\{3,6,9,10 \} \Bigl\}.
\end{align*}

One can easily check that $P_i$ is a strong domatic partition of $G_i$, for $1\leq i\leq 21$ and $i\neq 17$. So, we found a strong domatic partition  of size $3$ for each. Therefore we have the result.
\qed
	\end{proof}

As an immediate result of Corollary \ref{cor:strong-domatic-min-deg}, and Theorems \ref{thm:petersen} and \ref{thm:cubic10}, we have the following:

	\begin{corollary}
Domatic number and strong domatic number of the Petersen graph are unique among the cubic graphs of order $10$. 
	\end{corollary}


\end{document}